\documentclass[10pt]{article}

\usepackage[english]{babel}
\usepackage[utf8]{inputenc}
\usepackage{graphicx}

\usepackage{amsmath}
\usepackage{amssymb}
\usepackage{upgreek}
\usepackage{bbold} 

\usepackage{amsthm}

\usepackage{enumerate}

\newtheorem{theorem}{Theorem}[section]
\newtheorem{proposition}{Proposition}[section]
\newtheorem{lemma}{Lemma}[section]
\newtheorem{corollary}{Corollary}[section]
\newtheorem{definition}{Definition}[section]

\usepackage{color}

\begin{document}
\author{Matthieu Jonckheere, Manuel Sáenz}
\date{}
\title{Asymptotic optimality of degree-greedy discovering of independent sets in Configuration Model graphs}
\maketitle

\begin{abstract}
  Finding independent sets of maximum size in fixed graphs is well known to be an NP-hard task. Using scaling limits, we characterise the asymptotics of sequential degree-greedy explorations and provide sufficient conditions for this algorithm to find an independent set of asymptotically optimal size in large sparse random graphs with given degree sequences. In the special case of sparse Erdös-Rényi graphs, our results allow to give a simple proof of the so-called $e$-phenomenon identified by Karp and Sipser for matchings and to give an alternative characterisation of the asymptotic independence number.
  
\end{abstract}

\section{Introduction}

Given a graph $G=(V,E)$, an \emph{independent set} is a subset of vertices $A \subseteq V$ where no pair of vertices are connected to each other (i.e., for every pair $x, y \in A$ we have that $\{ x, y \} \notin E$). In the sequel, the number of vertices will be denoted by $n$.

Independent sets are relevant in the study of diverse physical and communication models. For example in physics, where dynamics that generate independent sets are used to study the deposition of particles in surfaces \cite{cadilhe2007random,evans1993random} as well as to model the number of excitations in ultracold gases \cite{sanders2015sub}. In the context of communications, similar stochastic processes where used \cite{bermolen2014estimating,dhara2016generalized} for theoretically modelling the possible number of simultaneous transmissions of information within a WiFi network. 

An independent set is said to be \emph{maximal} if it is not strictly contained in another independent set; and is said to be \emph{maximum} if there is no other independent set of greater size. Given a (finite undirected) graph $G$, the size of the maximum independent set(s) is called the \emph{independence number} and is usually represented with $\alpha(G)$. Finding a maximum independent set (or its size) in a general graph is known to be NP-hard\cite{frieze1997algorithmic}. For sparse random graphs, characterising the size of maximum independent sets and defining algorithms that find independent sets of (asymptotically, i.e., as $n$ diverges) maximum size are two questions that have received a lot of attention in the last decades but remain largely open. Before explaining our main results, we review the extended body of literature on this problem.

Mainly two types of approaches have been followed to obtain maximum size independent sets in random graphs: on the one hand, using reversible stochastic dynamics (usually Glauber dynamics); and on the other hand, using sequential algorithms. The Glauber dynamics consists of a reversible Markov dynamics on graphs where vertices become occupied (at a fixed rate called the activation rate) when none of its neighbours are; and, if occupied, become unoccupied at rate $1$. When the rate of activation tends to infinity, this dynamics are easily shown to concentrate on configurations being independent sets of maximum size. Though this is a very useful property, and not unlike for many other discrete optimisation problems, these dynamics might not be helpful in practice as the convergence towards a maximum size configuration can be extremely slow when the activation rate is large. In some special cases, the mixing time has been theoretically characterised. For example \cite{vigoda2001note}, when the degree distribution is bounded by $\Delta \geq 0$ and the activation rate is small ($\beta < \frac{2}{\Delta - 2}$) the mixing time is $\mathcal{O}(n \log(n))$. In the case of a bipartite regular graph, it was shown \cite{galvin2006slow} that for $\beta$ large enough the mixing time is actually exponential in $n$. Finally, this method does not allow to theoretically characterise the (asymptotical) independence number.

A completely different approach consists in defining algorithms that explore the graph sequentially (and hence terminate the exploration in less than $n$ steps). For this, at each step $k\geq0$ the vertex set is partitioned in three: the \emph{unexplored vertices} $\mathcal{U}_k$, the \emph{active vertices} $\mathcal{A}_k$ and the \emph{blocked vertices} $\mathcal{B}_k$. A typical sequential algorithm works as follows. Initially, it sets $\mathcal{U}_0 = V$, $\mathcal{A}_0 = \emptyset$ and $\mathcal{B}_0 = \emptyset$. To explore the graph, at the $k+1$-th step it selects a vertex $v_{k+1}\in\mathcal{U}_k$ (possibly taking into account its current or past degree towards other unexplored vertices), and changes its state into active. After this, it takes all of its unexplored neighbours, i.e.\ the set $\mathcal{N}_{v_{k+1}} := \{ w \in \mathcal{U}_k | (v_{k+1},w) \in E \}$, and changes their states into blocked. This means that, if in the $k+1$-th step vertex $v_{k+1}$ is selected, the resulting set of vertices will be given by $\mathcal{U}_{k+1} = \mathcal{U}_k \backslash \{ v_{k+1} \cup \mathcal{N}_{v_{k+1}} \}$, $\mathcal{A}_{k+1} = \mathcal{A}_k \cup \{ v_{k+1} \}$ and $\mathcal{B}_{k+1} = \mathcal{B}_k \cup \mathcal{N}_{v_{k+1}}$. Note that at each step, the set of active vertices defines an independent set. The algorithm keeps repeating this procedure until the step $k_n^*$ in which all vertices are either active or blocked (or equivalently $\mathcal{U}_{k_n^*} = \emptyset$). The set of active vertices at step $k_n^*$ then defines a maximal independent set.

In the \emph{greedy algorithm}, during the $k$-th step the activated vertex $v_k$ is selected uniformly from the subgraph of remaining vertices $G_k$ (i.e., the subgraph formed by the unexplored vertices). If $G$ is a graph, we will call $\sigma_{Gr}(G)$ the size of the independent set obtained by the greedy algorithm ran on $G$. This algorithm has been extensively studied, specially in the context of Erdös-Rényi graphs. There are many ways of approaching the problem of determining the asymptotic value of the independent set obtained by the greedy algorithm in a sparse Erdös-Rényi graph. For example, there have been \cite{grimmett1975colouring, karp1976probabilistic} combinatorial analysis of the problem. More akin to the rest of the paper, in \cite{Bermolen2017} an hydrodynamic limit for a one-dimensional Markov process associated to the algorithm was proved. The description of this exploration process in a Configuration Model cannot be described as a one-dimensional Markov process, as the unexplored vertices have degrees (towards other unexplored) that are not interchangeable and that depend in a complicated way on the evolution of the process. This makes the analysis much more involved that in the case of an Erdös-Rényi graph. There have been two works in the literature that describe an hydrodynamic limit for this process. In \cite{bermolen2017jamming}, an hydrodynamic limit for the degree distribution of the remaining graph was obtained. While in \cite{brightwell2016greedy}, a similar hydrodynamic limit was proved but for a modified dynamics that allows for a simplification of the limiting differential equations. Through these limits, the size of the independent set is determined.

The \emph{degree-greedy} algorithm is a variation of the greedy algorithm that takes into account the degree of the vertices in the remaining graph. During the $k$-th step an unexplored vertex $v_k$ is selected uniformly from the vertices of \emph{minimum degree} (towards other unexplored vertices) within the remaining subgraph $G_k$. If $G$ is a graph, we will denote by $\sigma_{DG}(G)$ the (possibly random) proportion of vertices in the independent set obtained by the degree-greedy algorithm ran on $G$. Although having been studied in the computer science community (for example, in \cite{halldorsson1997greed}), there are few exact mathematical results that characterise or bound the independent set found by this algorithm. A remarkable exception can be found in \cite{karp1981maximum,aronson1998maximum} which, although considering a completely different problem (namely, maximum matchings), imply that the degree-greedy algorithm is asymptotically optimal for Erdös-Rényi graphs when the mean degree is $\lambda < e$, a result coined as the $e$-phenomenon. Results by Wormald \cite{wormald1995differential} also describe an hydrodynamic limit for the process generated by the degree-greedy algorithm when run on a $d$-regular graph but without discussing asymptotic optimality.

\paragraph{Characterisation of maximum independent sets.} 
In the case of Erdös-Rényi graphs, their independence numbers are better understood in the case of large connection probabilities. In this case, a second moment argument \cite{bollobas1998random} yields its convergence in probability towards $- \frac{2 \log np}{\log q}$. In this same context, the relationship between the independence number and the asymptotic size of the independent set obtained by the greedy algorithm is also known to be $1/2$. In the case of sparse Erdös-Rényi graphs, as a consequence of the results proved in \cite{bollobas1976cliques,frieze1990independence}, similar (but weaker) results can be shown. This shows that the greedy algorithm has, for $G$ a sparse Erdös-Rényi graph of mean degree $\lambda$, a performance ratio $\sigma_{Gr}(G)/\alpha(G)$ that is $1/2$ asymptotically in $n$ for large $\lambda$. As mentioned before, it follows from \cite{karp1981maximum,aronson1998maximum} that the degree-greedy algorithm finds asymptotically a.s. maximum independent sets in Erdös-Rényi graphs of mean degree smaller than $e$. In the same works, they also characterise the asymptotic independence number of these graphs as the combination of the roots of certain functions.

The proof of the existence of a limiting independence ratio for random $d$-regular graphs was given in \cite{bayati2010combinatorial}. In \cite{lauer2007large}, Wormald proved a lower bound for the independence number of a $d$-regular graph. While in the same work, Wormald (and Gamarnik and Goldberg independently in \cite{gamarnik2010randomized}) showed that the proportion of vertices in the independent set found by a greedy algorithm in a $d$-regular graph is (for $d\geq3$), asymptotically (in probability) when $n\rightarrow\infty$ and the girth\footnote{Length of the shortest cycle.} $g \rightarrow\infty$, given by a certain function of $d$. Moreover, bounds were also proved in \cite{bollobas1981independence,mckay1987lnl} and an alternative lower bounds in \cite{frieze1992independence}. Finally, in a recent work \cite{ding2016maximum}, the exact large number law for the independence ratio of regular graphs of sufficiently large $d$ was established as the solution of a polynomial equation.

\paragraph{Contribution.} 
Both characterisation of independence numbers  and algorithms to discover independent sets in (deterministic) polynomial times for large sparse random graphs, e.g. in Configuration Models, are still open problems for most cases (the only exceptions being the characterisations of the independence numbers of regular graphs of large enough degree and the optimality of degree-greedy for sparse Erdös-Rényi graphs with $\lambda < e$). 

Using scaling limits on sequential explorations that select only degree 1 vertices, we decompose the exploration in different steps and we show that these steps can be described as a combination of two maps acting on the degree-distribution of the graph. Using these results, we show that for a large class of sparse random graphs, a degree-greedy exploration is actually asymptotically optimal. We first give a sufficient condition, which can be easily verified in practice, for this exploration to be optimal \emph{in one step}. We then show how to generalise this sufficient condition by characterising the remaining graph after \emph{several steps}. Finally, we study the case of Poisson distribution (which is asymptotically equivalent to Erdös-Rényi graphs) and show, in an alternative way, that when the mean degree is smaller than $e$ the exploration is asymptotically optimal. We use this fact to give a new characterisation of the independence number under these circumstances. 

We now state rigorously our results and give various examples.


\section{Main results}

We first state lemmas for deterministic graphs. Afterwards, we enunciate optimality results for random graphs when the hydrodynamic limit of the degree-greedy sequential exploration selects only degree 1 vertices. We then show that this property is actually satisfied by a large class of random graphs (including all strictly subcritical graphs, in the connectivity sense). We finally proceed to characterise this class of graphs.

\subsection{First characterisation of degree-greedy asymptotic optimality}

\paragraph*{Criterion for deterministic graphs}

In Lemma \ref{lemma:optimal} below, we show that any algorithm that selects at each step of its implementation a vertex of degree 1 (or 0) is optimal in the sense that it finds a maximum independent set. An analogous version of this lemma for matchings is stated in \cite{karp1981maximum}.

We introduce now some definitions to state this more precisely.

\begin{definition} 
    Given a graph $G = (V,E)$, we call a finite sequence $W=\{w_1,w_2,\ldots,w_m\}$ ($m\leq n$) of distinct vertices of $V$ a \textbf{selection sequence } (of $G$) if no vertex in $W$ is neighbour to another vertex in $W$ and every vertex in $V$ is either in $W$ or neighbour of a vertex in $W$.
\end{definition}

Note that the conditions in this definition ensure that the vertices in $W$ define a maximal independent set. By definition, sequential algorithms define random selection sequences.

\begin{definition}\label{def:remaining}
 Let $W=\{w_1,...,w_m\}$ be a selection sequence. Then, for every $1 \leq i \leq m$, we denote by \textbf{$i$-th remaining subgraph}, the subgraph formed by the vertices that are neither in $\{w_1,...,w_i\}$ nor neighbours of any of them. We denote it by $G_i$ and we define $G_0 := G$.
\end{definition}

When there is no ambiguity to which value of $i$ the remaining graph corresponds to, we just call it the \emph{remaining graph}. When analysing the degree-greedy algorithm, the remaining graphs will refer to the remaining graphs with respect to the selection sequence defined by the algorithm. Of course, as a selection sequence $W = \{ w_1,...,w_m\}$ always determines a maximal independent set, $G_m = \emptyset$.

The degree-greedy algorithm run on a finite graph $G$ can be thought of as a random selection sequence $W_{DG}$, built inductively in the following manner: given $\{w_1,...,w_k\}$ the first $k \geq 1$ vertices of $W_{DG}$, $w_{k+1}$ is a vertex chosen uniformly from the lowest degree vertices of $G_k$.

\begin{definition}
    Let $W$ be a selection sequence. We say that $W$ \textbf{has the property $T_1$} if for every $1\leq i \leq m$ the degree of $w_i$ in $G_{i-1}$ is equal or less than $1$.
\end{definition}

Then, a selection sequence has the property $T_1$ if at each step it \emph{selects} a vertex that has degree either 0 or 1 in the corresponding remaining graph.

We are now in a position to state our first lemma.
\begin{lemma}\label{lemma:optimal}
Let $G$ be a finite graph and $W$ be a selection sequence of $G$. Then, if $W$ has the property $T_1$, $|W| = \alpha(G)$.
\end{lemma}

\paragraph*{$T_1$ property and asymptotic optimality of the degree-greedy exploration.}
From now on, we use the usual \emph{big $\mathcal{O}(\cdot)$ and little $o(\cdot)$} notation to describe the asymptotic behaviour of functions of the graph size $n$. We also use the \emph{probabilistic big $\mathcal{O}_{\mathbb{P}}(\cdot)$ and little $o_{\mathbb{P}}(\cdot)$} notation in the following sense:

\begin{itemize}
    \item A sequence of random variables $X_n$ is $\mathcal{O}_{\mathbb{P}}(f_n)$ (for some function $f_n : \mathbb{N} \rightarrow \mathbb{R}_{>0}$) if for every $\epsilon > 0$ there exists $M > 0$ and $N \in \mathbb{N}$ s.t. $\mathbb{P}(|X_n / f_n| > M ) < \epsilon$ for every $n \geq N$.
    \item Conversely, a sequence of random variables $X_n$ is $o_{\mathbb{P}}(f_n)$ if $X_n/f_n \xrightarrow{\mathbb{P}} 0$ as $n \rightarrow \infty$.
\end{itemize}

We will also say that an event holds \emph{with high probability} (w.h.p.) whenever its probability is some function $1+o(1)$ of the graph size.

Throughout the paper we will consider random graphs with given degrees, (a.k.a., Configuration Models \cite{van2016random}). In this construction, given a \emph{degree sequence} $\bar d^{(n)} \in \mathbb{N}_0^n$ (which could either be a fixed sequence or a collection of i.i.d. variables), we form an $n$-sized \emph{multigraph}\footnote{That is, a graph where there are possibly edges between a vertex and itself (\emph{selfedges}) and multiple edges between a pair of vertices (\emph{multiedges}).} with degrees $\bar d^{(n)}$ in the following manner: first assign to each vertex $v \in \{1,...,n\}$ a number $d_v$ of half-edges, then sequentially match uniformly each half-edge with another unmatched one, and finally for every pair of vertices in the multigraph establish an edge between them for every pair of matched half-edges they share. The distribution of the random multigraphs generated according to this matching procedure will be denoted by $\mbox{CM}_n(\bar d^{(n)})$. Because all the matchings are made uniformly, the resulting multigraph is equally distributed no matter in which order the half-edges are chosen to be paired \cite{van2016random}. This fact allows, when analysing a process in the graph, for the matching to be incorporated into the dynamics in question.

In the remaining of the paper we will consider Configuration Model graphs that obey the following assumption: 

\vspace{0.5cm}
\noindent
\textbf{Convergence assumption (CA):} \emph{when dealing with sequences of graphs $(G_n)_{n\geq1}$ with $G_n \sim \mbox{CM}_n(\bar d^{(n)})$, we will always assume that $D^{(n)} \xrightarrow{\mathbb{P}} D$ and $\mathbb{E}(D^{(n)2}) \xrightarrow{n\rightarrow\infty} \mathbb{E}(D^2)$. Where $D^{(n)}$ is the r.v. that gives the degree of a uniformly chosen vertex in $G_n$ and $D$ is a random variable with finite second moment. We will also use $(p^{(n)}_k)_{k\geq0}$ to refer to the distribution of $D^{(n)}$ and $(p_k)_{k\geq0}$ to the one of the asymptotic degree random variable $D$.}

\vspace{0.5cm}

When there is no ambiguity, we will omit the subindex in $G_n$. 

Although this construction results in a multigraph rather than a \emph{simple graph}\footnote{One with no self nor multiedges.}, as showed in \cite{janson2009probability}, because under the (CA) there is asymptotically a probability bounded away from 0 of obtaining a simple graph. This means that any event that has been shown to hold w.h.p. for $G \sim \mbox{CM}_n(\bar d^{(n)})$, can be automatically showed to hold w.h.p. for the construction conditioned to result in a simple graph.

Another important feature of the model is that its largest connected component asymptotically contains a positive proportion of the vertices of the graph iff $\nu := \mathbb{E}(D(D-1))/\mathbb{E}(D) > 1$ \cite{molloy1998size,janson2009new} (this quantity will be referred as the \emph{criticality parameter} of the graph). When this condition holds, we will say that the graph is \emph{supercritical}; and when it does not, that it is \emph{subcritical}.

We can now state a sufficient condition for the degree-greedy algorithm to find w.h.p. an independent set that asymptotically contains the same proportion of vertices as a maximum one: 

\begin{proposition}[Sufficient condition for the near optimality of the degree-greedy in CM]\label{prop:optimalidadCM} 
    Let $G \sim \mbox{CM}_n(\bar d^{(n)})$ be a sequence of graphs distributed according to the CM, and assume that the limiting $(p_k)_{k\in\mathbb{N}}$. If the degree-greedy algorithm defines w.h.p. a selection sequence that selects only vertices of degree 1 or 0 until the remaining graph is subcritical and has a degree distribution that is $\mathcal{O}(e^{-\gamma k})$ (for some $\gamma > 0$), then (for every $\alpha > 0$) $\sigma_{DG}(G) = \alpha(G) + \mathcal{O}_{\mathbb{P}}(n^\alpha)$.
\end{proposition}    
    
This is so because a subcritical graph looks like (up to sufficiently small differences) a collection of trees. We can then couple the algorithm running in the subcritical graph with one running in the collection of spanning trees of its components. This coupling will only differ in the components that are not trees which, as shown in Proposition \ref{prop:2coreChico}, contain (for every $\alpha > 0$) $\mathcal{O}_{\mathbb{P}}(n^\alpha)$ vertices if the limiting degree distribution has an exponentially thin tail. Therefore, both algorithms find an independent set of roughly the same size.

\subsection{Further characterisations of asymptotic optimality}

We now give explicit criteria to show that a given limiting distribution meets the hypothesis in Proposition \ref{prop:optimalidadCM}. We first state a criterion that can be easy to handle in practice. We then refine this criterion and give a general way of characterising the degree distributions in question.

\paragraph*{One application of the map $M_1$.}\label{sec:oneapplication}

To characterise when the degree-greedy algorithm does only select vertices of degree 1 or 0 until the remaining subgraph is subcritical, it will be useful to break the evolution of the process into discrete intervals of time for which we know for sure that the only vertices selected have these degrees. 

The key observation is that if the graph initially has $n p_1 + o(n)$ vertices of degree 1\footnote{For simplicity, vertices of degree 0 will be omitted from the analysis because, when selected, they block no vertices and then do not modify the number of unexplored vertices of other degrees. We can think that the algorithm selects them immediately after they are produced.}, then the degree-greedy will select vertices of degree 1 at least until an equivalent number of degree 1 vertices have been explored. Then we define the map $M_1^{(n)} : \mathbb{R}_{\geq0}^{\mathbb{N}} \longrightarrow \mathbb{R}_{\geq0}^{\mathbb{N}}$ as the map that when evaluated in a degree distribution of an $n$ sized graph $(p_k^{(n)})_{k\geq0}$ gives the resulting normalised (by $n$) degree measure of unexplored vertices after $n p_1^{(n)}$ vertices of degree 1 have been activated or blocked (and their neighbours blocked). This is in principle a stochastic map but, as we will prove, the degree-greedy exploration converges to a deterministic limit which implies that $M_1^{(n)}(\cdot)$ also behaves as a deterministic limit map $M_1(\cdot)$.

In this section, we determine when the degree distribution obtained after one application of the map $M_1^{(n)}$ is w.h.p. subcritical, and therefore the hypothesis of Proposition \ref{prop:optimalidadCM} are met. Our main result here is the following theorem:

\newpage

\begin{theorem}\label{thm:oneapplication}
 Given $G \sim \mbox{CM}_n(\bar d^{(n)})$ where the $(CA)$ holds towards a limiting degree distribution $(p_k)_{k\geq0}$ of mean $\lambda > 0$ and finite second moment. If
 
  \begin{equation}
     \tilde \nu : = G_D''(Q) / \lambda < 1
  \end{equation}

where $Q:=(1 - p_1/\lambda)$ and $G_D(z)$ is the generating function of the asymptotic degree r.v. $D$; then, (for every $\alpha > 0$) $\sigma_{DG}(G) = \alpha(G) + \mathcal{O}_\mathbb{P}(n^\alpha)$.
\end{theorem}

\paragraph*{Further applications of $M_1$.}\label{sec:manyapplications}

Theorem \ref{thm:oneapplication} establishes an asymptotic condition for the remaining graph obtained after one application of the map $M_1^{(n)}(\cdot)$ to be subcritical, and thus for the degree-greedy to be asymptotically optimal (in the sense that it finds, asymptotically, an independent set with the same proportion of vertices as a maximum one). Here we compute the asymptotic degree distribution of the remaining graph after one application of $M_1^{(n)}(\cdot)$ and in doing so we allow for the study of further applications of $M_1^{(n)}(\cdot)$. This can be used to establish more general conditions that determine the asymptotic optimality of the degree-greedy algorithm. For doing so, we determine an hydrodynamic limit for the second phase of $M_1^{(n)}(\cdot)$ and solve the obtained equations.

\begin{theorem}\label{thm:furtherapp} 
Define (for every $i, j \geq 1$) $\eta_j(i) := (-1)^{j-i} \binom{j}{i} \mathbb{1}_{i\leq j}$. Then, under the same assumptions of Theorem \ref{thm:oneapplication} and if we call $(a_j)_{j\in\mathbb{N}}$ the components of $(Q^k p_k \mathbb{1}_{\{k\geq2\}})_{k\in\mathbb{N}}$ in the base $\{\eta_j(\cdot)\}_{j\in\mathbb{N}}$, we have that the remaining graph after one application of the map is a Configuration Model graph with normalised degree measure given by

\begin{equation*}
    M^{(n)}_1\left(p^{(n)}_k\right)(i) \xrightarrow{\mathbb{P}}  M_1\left(p^{(n)}_k\right)(i) := \sum_{j \geq i} a_j (-1)^{j-i} \tilde Q^j \binom{j}{i}, \ \mbox{for } i \geq 1
\end{equation*}

where $\tilde Q := \sum_{i\geq2}i Q^i p_i / Q^2 \lambda$.
 
\end{theorem}

As mentioned above, this result gives a way of generalising the condition for asymptotic optimality of Theorem \ref{thm:oneapplication}:

\vspace{0.5cm}
\noindent
\textbf{General criterion for asymptotic optimality (GC):} \emph{given a limiting degree distribution, if after a finite number of applications of the map $M_1(\cdot)$ the degree distribution obtained is subcritical, then the degree-greedy is asymptotically optimal for a Configuration Model graph with that limiting distribution.}
\vspace{0.5cm}

The proof of this criterion is a direct consequence of Proposition \ref{prop:optimalidadCM}. Then, Theorem \ref{thm:furtherapp} can be used to verify it. In the next section we give numerical computations of distributions that meet the criterion.

\subsection{Erdös-Rényi graphs}

Here we analyse the special case of graphs with asymptotic Poisson degree distributions. They are of particular importance because, by \cite{kim2007poisson}, asymptotic results for them can be directly extended to Erdös-Rényi random graphs. We give here an alternative and simpler proof of the so-called ``$e$-phenomenon'', identified for matchings in \cite{karp1981maximum,aronson1998maximum}. We also give a more explicit characterisation of the asymptotic independence number than the one present in these works: we prove that both constants in the expression obtained in \cite{aronson1998maximum} are in fact the same and characterise them in terms of the Lambert function.

\begin{corollary}
    Let $G \sim \mbox{ER}_n(\lambda)$. If $\lambda < e$, then $\sigma_{DG}(G) = \sigma(G) + o_{\mathbb{P}}(n)$; otherwise, the selection sequence does not have the property $T_1$. Furthermore, in this case, $\alpha(G) = n (z(\lambda) + \frac{\lambda}{2} z(\lambda)^2) + o_{\mathbb{P}}(n)$, where $z(\lambda):= e^{-W(\lambda)}$ with $W(x)$ the Lambert function.
\end{corollary}

\begin{proof}
Because of the following equality for Bernstein polynomials \cite{lorentz2012bernstein},

\begin{equation*}
    \binom{n}{i} x^i (1-x)^{n-i} = \sum_{j=i}^n \binom{n}{j} x^j \eta_j(i),
\end{equation*}
the expansion in the base $\{\eta_k(\cdot)\}_{k\in\mathbb{N}}$ can be explicitly computed for binomial distributions, which in turn gives the transformation for Poisson distributions taking the usual limit. Using this, it can be easily seen that after $i$ applications of the map $M_1(\cdot)$ to a Poisson distribution, the resulting distribution is a linear combination of a Poisson distribution of mean $\mu_i$ and a term $\delta_1(\cdot)$ (we here ignore degree 0 vertices as they do not play a role in the dynamics), with respective coefficients $A_i$ and $B_i$. 

This transformation can be used to derive the following recursion relation for $\lambda_i$, $A_i$ and $B_i$:
\begin{equation*}
    \begin{cases}
        \mu_{i+1} = \tilde Q_i Q_i \mu_i, \\
        A_{i+1} = e^{-(1-Q_i)\mu_i} A_i, \\
        B_{i+1} = - A_i \tilde Q_i Q_i \mu_i e^{- \mu_i},
    \end{cases}
\end{equation*}
where $Q_i$ and $\tilde Q_i$ are the corresponding coefficients defined in Theorems \ref{thm:oneapplication} and \ref{thm:furtherapp} for the distribution after the $i$-th application. Writing explicitly the coefficients $Q_i$ and $\tilde Q_i$, one arrives at the 3-dimensional iterative map
\begin{equation}\label{eq:mappois}
    \begin{cases}
        \mu_{i+1} = \left( e^{-\frac{a_i \mu_i}{\lambda_i}} - e^{-\mu_i} \right) \frac{A_i \mu_i^2}{\lambda_i}, \\
        A_{i+1} = e^{-\frac{a_i \mu_i}{\lambda_i}} A_i, \\
        B_{i+1} = - e^{-\mu_i} A_i \mu_{i+1},
        \end{cases}
\end{equation}
where $a_i := A_i \mu_i e^{-\mu_i}+ B_i$ is the number of remaining degree 1 vertices and $\lambda_i := A_i \mu_i+ B_i$ is the number of remaining edges. Furthermore, the coordinate $B_i$ can be eliminated arriving at the closed 2-dimensional iterative map
\begin{equation*}
    \begin{cases}
        \mu_{i+1} =  \left( e^{-\frac{A_i e^{-\mu_i}-A_{i-1} e^{-\mu_{i-1}}}{A_i-A_{i-1} e^{-\mu_{i-1}}}\mu_i} - e^{-\mu_i} \right) \frac{A_i \mu_i}{A_i - A_{i-1}e^{-\mu_{i-1}}}, \\
        A_{i+1} =  e^{-\frac{A_i e^{-\mu_i}-A_{i-1} e^{-\mu_{i-1}}}{A_i-A_{i-1} e^{-\mu_{i-1}}}\mu_i} A_i.
        \end{cases}
\end{equation*}
Making the coordinate change $v_i := e^{-\mu_i}$ and $w_i := v_i^{A_i/(A_i-v_{i-1} A_{i-1})}$ the map can be rewritten as
\begin{equation}\label{eq:mapdef}
    \begin{cases}
        w_{i+1} = w_i^{\left(v_i w_i^{v_i-1}\right)}, \\
        v_{i+1} = w_{i+1} w_i^{-v_i},
        \end{cases}
\end{equation}
where the initial conditions for this discrete system are $v_0 = w_0 = e^{-\lambda}$. We now show that for initial conditions in the identity line $(v_0, v_0)$ this recursion can be explicitly solved. We first show by induction that in this case
\begin{equation*}
    w_i = h_{2i+1}(v_0),
\end{equation*}
\begin{equation*}
    v_i =  \frac{h_{2i+1}(v_0)}{h_{2i}(v_0)}, 
\end{equation*}
where we defined $h_m(a) := \underbrace {a^{a^{\cdot ^{\cdot ^{a}}}}} _{m}$; that is, $h_m(a)$ is the \emph{$m$-th tetration} of $a$. First note that by applying the map \eqref{eq:mapdef} one obtains that $w_1 = v_0^{(v_0^{v_0})} = h_{3}(v_0)$ and $v_1 = w_1 v_0^{-v_0} = h_{3}(v_0) h_{2}(v_0)^{-1}$, thus starting the induction. To advance the induction, suppose $w_i = h_{2i+1}(v_0)$ and $v_i = h_{2i+1}(v_0) \ h_{2i}(v_0)^{-1}$. Then, applying the map one obtains that
\begin{equation*}
    \begin{split}
    w_{i+1} & = h_{2i+1}(v_0)^{\left(h_{2i+1}(v_0) h_{2i}(v_0)^{-1} h_{2i+1}(v_0)^{\left(h_{2i+1}(v_0) h_{2i}(v_0)^{-1}-1\right)}\right)},\\ 
    & = h_{2i+1}(v_0)^{\left(h_{2i}(v_0)^{-1} h_{2i+1}(v_0)^{h_{2i+1}(v_0) h_{2i}(v_0)^{-1}}\right)}, \\
    & = \left( h_{2i+1}(v_0)^{h_{2i}(v_0)^{-1}} \right)^{\left( h_{2i+1}(v_0)^{h_{2i}(v_0)^{-1}} \right)^{h_{2i+1}(v_0)}}, \\
    & = v_0^{\left(v_0^{h_{2i+1}(v_0)}\right)} = h_{2(i+1)+1}(v_0),
    \end{split}
\end{equation*}
where we used that $h_{2i+1}(v_0)^{h_{2i}(v_0)^{-1}} = v_0$. In an analogous way, we can show that $v_{i+1} = \frac{h_{2(i+1)+1}(v_0)}{h_{2(i+1)}(v_0)}$, advancing the induction.

By Theorem 5 in \cite{barrow1936infinite}, $h_{i}(v_0)$ converges if and only if $v_0 \in (e^{-e}, e^{1/e})$. And in these cases, because of \eqref{eq:mapdef}, $v_i$ will converge to $1$ as $i \rightarrow \infty$. Undoing the coordinate change, this implies that if $ \lambda < e$, we will have that $\mu_i \xrightarrow{i\rightarrow\infty} 0$. This in turn proves that the number of remaining vertices in the graph $(1- e^{-\mu_i}) A_i + B_i \xrightarrow{i\rightarrow\infty} 0$ which means that the graph vanishes asymptotically, showing that under these conditions the degree-greedy algorithm is asymptotically optimal. Furthermore, also by Theorem 5 in \cite{barrow1936infinite}, if $\lambda > e$, $h_{2i}(v_0)$ and $h_{2i+1}(v_0)$ will converge to different limits yielding $v_i \xrightarrow{i\rightarrow\infty} t < 1$. That is, in these cases, the selection sequence defined by the degree-greedy algorithm will not (w.h.p.) have the property $T_1$.

Moreover, making use of Proposition \ref{prop:independence} from next section, we can compute the independence number. Because in this case the number of degree 2 or greater vertices is given by a Poisson density of mean $\mu_i$ multiplied by $A_i$,

\begin{equation*}
    \sum_{j\geq2} (1-Q_i^j) A_i \frac{\mu_i^j}{j!}e^{-\mu_i} = A_i(1-e^{-\mu_i(1-Q_i)}) - A_i \mu_i e^{-\mu_i} (1 - Q_i)
\end{equation*}

Also, using the explicit solution obtained for the map \eqref{eq:mappois}, it can be shown that 
\begin{itemize}
    \item $\ \ A_i = h_{2i}(e^{-\lambda})$
    \item $\ \ Q_i = \frac{( h_{2i}(e^{-\lambda}) - h_{2i+1}(e^{-\lambda}) )}{( h_{2i}(e^{-\lambda}) - h_{2i-1}(e^{-\lambda}) )}$
    \item $\ \ \mu_i = \lambda ( h_{2i}(e^{-\lambda}) - h_{2i-1}(e^{-\lambda}) )$
    \item $\ \ a_i = \lambda ( h_{2i+1}(e^{-\lambda}) - h_{2i-1}(e^{-\lambda}) ) ( h_{2i}(e^{-\lambda}) - h_{2i-1}(e^{-\lambda}) )$
\end{itemize}

Where $a_i$ is the normalized remaining number of degree 1 vertices. Using this, Proposition \ref{prop:independence} gives

\begin{equation*}
    \begin{split}
        \alpha(G) & = n \left(1 - \sum_{i\geq0} h_{2i}(e^{-\lambda}) -h_{2i+2}(e^{-\lambda})  + \sum_{i\geq0} \frac{\lambda}{2} h_{2i-1}^2(e^{-\lambda})  - \frac{\lambda}{2} h_{2i+1}^2(e^{-\lambda})  \right) + o_{\mathbb{P}}(n) \\
        & = n ( z(\lambda) + \lambda/2 z(\lambda)^2 ) + o_{\mathbb{P}}(n) 
    \end{split}
\end{equation*}

Where we used that (for $m\in\mathbb{N}$ and $a>0$) $h_{-m}(a) = 0$, $h_0(a) = 1$ and that $\lim_{i\rightarrow\infty} h_i(a) = e^{-W(-\log(a))}$.

\end{proof}

\subsection{Other applications}\label{sec:exandapp}

We now apply our results to power-law distributions, which give rise to scale-free networks. Furthermore, we explain how to compute the independence ratios for this distributions and we prove upper bounds for the independence ratios of distributions for which these theorems cannot be used.

\paragraph{Power-law distributions.}

Here we look at the case where the degree distribution obeys a power law of parameter $\alpha > 3$. Because the generating function of a power law distribution $p_k = C_\alpha k^{-\alpha}$ is given by $C_\alpha \mbox{Li}_\alpha(z)$ (where $\mbox{Li}_\alpha(z)$ is the polylogarithm function of order $\alpha$):

\begin{equation*}
 \tilde \nu(\alpha) = \frac{(C_\alpha \mbox{Li}_\alpha(z))''|_{Q(\alpha)}}{\sum_{i\geq1} i C_\alpha i^{-\alpha}} = \frac{\mbox{Li}_{\alpha-2}(1-\zeta(\alpha-1)^{-1})-\mbox{Li}_{\alpha-1}(1-\zeta(\alpha-1)^{-1})}{\zeta(\alpha-1)}
\end{equation*}

Where $\zeta(z)$ is the Riemann zeta function and in the last line we used that $Q(\alpha) = 1 - \zeta(\alpha-1)^{-1}$ and that $\mbox{Li}_{\alpha}(z)' = \mbox{Li}_{\alpha-1}(z)/z$. This last expression can be seen to be smaller than $1$ for every $\alpha > 3$; which means that for every power law distribution of finite second moment, the degree-greedy algorithm is asymptotically optimal. In particular, whenever it has finite second moment and $\zeta(\alpha-2) > 2 \zeta(\alpha-1)$ (or $3 < \alpha \lesssim 3,478$), this distribution will be supercritial but nevertheless the degree-greedy will be asymptotically optimal.

\paragraph{Computing independence ratios.}
As a consequence of Theorem \ref{thm:furtherapp}, whenever a sequence of graphs is under its hypothesis and after a finite number of applications of the map $M_1(\cdot)$ a subcritical distribution is obtained, the asymptotic independence number can be obtained by computing the number of vertices in the independent set constructed by the degree-greedy algorithm.

\begin{proposition}\label{prop:independence}
 Given $G \sim \mbox{CM}_n(\bar d^{(n)})$ where the $(CA)$ holds towards a limiting degree distribution $(p_k)_{k\geq0}$. Then, if the criterion $(GC)$ holds, we will have that
\begin{equation}\label{eq:indnumestimation}
    \alpha(G) = n \left( 1 - \sum_{i=1}^\infty  \frac{\mu^{(i)}(1) (1-Q_i)}{2} + \sum_{j=2}^\infty (1-Q_i^j) \mu^{(i)}(j)\right) + o_\mathbb{P}(n).
\end{equation}

Where $\mu^{(i)}(j)$ is the remaining number of degree $j$ vertices over $n$ and $Q_i$ is the corresponding parameter defined in Theorem \ref{thm:oneapplication}, after $i$ applications of the map $M_1(\cdot)$ over the limiting degree distribution of the graph sequence.
\end{proposition}

\begin{proof}
    During the $i$-th application of the map $M_1(\cdot)$, the number of vertices of degree $j$ (where $j\geq2$) that gets blocked\footnote{Because, as we will see in Section \ref{sec:proof22}, $Q_i$ is the probability of a single half-edge to be connected to an activated vertex.} is $(1-Q_i^j)\mu^{(i)}(j)$. Moreover, for each pair of degree $1$ vertices that get connected to another degree $1$ vertex, one gets blocked. The number of degree $1$ vertices that get blocked is then $(1-Q_i)\mu^{(i)}(1)/2$. So, the size of the independent set obtained will be 1 minus the total number of blocked vertices. Also, because we are assuming we are in the condition of asymptotic optimality of the algorithm, the size of the independent set obtained will be the independence number of the graph.
\end{proof}

\paragraph{Upper bounds for independence ratios.}
For sequences of graphs where the asymptotic optimality condition does not hold, one can nevertheless use Theorem \ref{thm:oneapplication} to construct upper bounds on the limiting independence number.

\begin{proposition}
 Given $G \sim \mbox{CM}_n(\bar d^{(n)})$ where the $(CA)$ holds towards a limiting degree distribution $(p_k)_{k\geq0}$; if we define $c_1 := \inf \{ a > 0 : G''_D(1- (a+p_1)/(a+\lambda)) < 1 \}$ and (for $k\geq0$) $p_k^*:= (c_1 \delta_{1k} + p_k)/(c_1+1)$, we will have that
 
 \begin{equation*}
     \alpha(G) \leq n \left( 1 - \sum_{i=1}^\infty  \frac{\mu^{*(i)}(1) (1-Q^*_i)}{2} + \sum_{j=2}^\infty (1-Q_i^{*j}) \mu^{*(i)}(j)\right) + o_\mathbb{P}(n).
 \end{equation*}
 
 Where $\mu^{*(i)}(j)$ is the remaining number of degree $j$ vertices over $n$ and $Q^*_i$ is the corresponding parameter defined in Theorem \ref{thm:oneapplication}, after $i$ applications of the map $M_1(\cdot)$ over the degree distribution $(p^*_k)_{k\geq0}$.
\end{proposition}

\begin{proof}
    Define a new graph sequence $G^* \sim \mbox{CM}_n(\bar d^{*(n)})$ where the $(CA)$ holds towards the limiting degree distribution $(p^*_k)_{k\geq0}$. This sequence can be thought as the graphs formed by changing the degrees of a certain proportion of vertices in the original graphs $G$ to 1 so that the inequality in Theorem \ref{thm:oneapplication} holds. The resulting graphs will be distributed as the original ones with some of its edges removed. Then, because the independence number is monotonic on edge removal, we will have that $\alpha(G) \leq \alpha(G^*)$. Finally, using equation \eqref{eq:indnumestimation}, $\alpha(G^*)$ can be computed and thus the desired upper bound is obtained.
\end{proof}

\section{Proofs}
In this Section, we provide proofs for all our results.

\subsection{Proof of Lemma \ref{lemma:optimal} on deterministic graphs}

We will prove Lemma \ref{lemma:optimal} by induction on the graph size $|G|$. For $|G|=1$ and $|G|=2$ the statement is trivially true. 

Now assume that it holds for $|G|=n$, we will show that it is thus also true for $|G|=n+1$. Suppose that it is not the case that $|W| = \alpha(G)$, then there is an independent set $A$ such that $|W|<|A|=\alpha(G)$. Calling $W'=(w_i)_{i=2}^{|W|}$, by the induction hypothesis (because $W'$ has the $T_1$ property in $G_1$) we know that $|W'|=\alpha(G_1)$. 

Calling $n_1$ the vertex adjacent to $w_1$,\footnote{Here we will assume that $d_{w_1}=1$. The proof is very similar in the case where $w_1$ is an isolated vertex.} because the independent set $A$ is of maximum size, it has to contain either $w_1$ or $n_1$ (if not, one could construct an even larger independent set by adding $w_1$ to the vertices in $A$, which would be a contradiction). This implies that
\begin{equation*}
    |A|= |A \backslash \{w_1,n_1\}| + 1 > |W| = |W'| + 1,
\end{equation*}
which means that $|A \backslash \{w_1,n_1\}| > |W'|$ which is a contradiction because $A \backslash \{w_1,n_1\}$ defines an independent set of $G_1$ and by hypothesis we have that the independent set defined by $W'$ is an independent set of $G_1$ of maximum size. We then have that $W$ defines an independent set of maximum size of $G$, advancing in this way the induction.

\begin{corollary}\label{coro:arbol}
 If $H$ is a collection of trees, then the degree-greedy algorithm run on $H$ finds a.s. a maximum independent set.
\end{corollary}

\begin{proof}
 Because $H$ is a collection of trees, for every leaf removed by the degree-greedy algorithm, further leafs (or isolated vertices) will be created. The algorithm will have the property $T_1$ as it will only select leafs (or isolated vertices). The conclusion is then reached by Lemma \ref{lemma:optimal}.
\end{proof}

\subsection{Proof of Proposition \ref{prop:optimalidadCM}}

We first deal with the case where $G$ is a subcritical graph.

To show the condition for optimality of Proposition \ref{prop:optimalidadCM}, we will study the number of times a breadth-first exploration process of the components joins two already explored vertices (for more details on the breadth-first exploration of components see \cite{van2016random}) forming a loop. Here we will call $N_u$ the number of times this happens during the exploration of the connected component associated to the vertex $u \in V$ and $N$ the total number of times it happens in the exploration of all the components in the graph. For every component $C(u)$ we will have that if $N_u = 0$, then the component is exactly a tree (as no loops are formed during the exploration process). The idea will be to prove that in a subcritical graph with a degree distribution with an exponentially thin tail, almost every vertex is in a component that is a tree. Note that this is not a consequence of the results in \cite{janson2007simple}, as it is not enough to show that the 2-core\footnote{The 2-core of a graph is the maximum subgraph with minimum degree 2.} is $o_\mathbb{P}(n)$ to conclude this.
\\\\
By calling $v(i)$ the vertex visited during the $i$-th step of the exploration process of the connected component of vertex $1$ and $T$ the stopping time in which the process finishes, we can write
\begin{equation}
 N_1 = \sum_{i=1}^{T} \sum_{j>i}^{T}\mathbb{1}_{\{v(i)=v(j)\}}
\end{equation}

We can now state the following bound on the conditional expectation of $N_1$ that we will use to prove our lemma.

\begin{lemma}\label{lemma:cotacond}
Let $G \sim \mbox{CM}_n(\overline{d})$ where the $(CA)$ holds towards a limiting degree distribution $(p_k)_{k\geq0}$ of mean $\lambda > 0$ and criticality parameter $0<\nu<\infty$. If $C(1)$ is the connected component of vertex $1$, then conditioning on the number $T$ of edges in the component, we have that

\begin{equation*}
    \mathbb{E}\left(N_1|T\right) \leq \frac{(\nu^{(n)} + 1) n \lambda^{(n)} T(T-1)}{2 (n \lambda^{(n)} -T)^2}
\end{equation*}

where $\lambda^{(n)}:= \sum_{u\in V^{(n)}} d_u^{(n)}/n \xrightarrow{n\rightarrow\infty}\lambda$ and $\nu^{(n)} := \sum_{u\in V^{(n)}} d_u^{(n)}(d_u^{(n)}-1) / n \lambda^{(n)}\xrightarrow{n\rightarrow\infty}\nu$.
\end{lemma}

\begin{proof}
The bound is derived making use of the corresponding bound for the probability that (conditional on $T$) a particular vertex is visited by the exploration process at time $i$. So, firstly we want to show that for $u \in V^{(n)}$ and $i \leq T$

\begin{equation}\label{eq:desprobcond}
    \mathbb{P}(v(i)=u|T) \leq \frac{d_u^{(n)}}{n\lambda^{(n)}-T}
\end{equation}

This can be shown using that the random variable $T$ is a.s. positive and bounded by $n^2$. Then for any set $A$ $\sigma(T)$-measurable we have that
\begin{align*}
    \mathbb{E}(\mathbb{1}_A \mathbb{1}_{\{ v(i) = u \}}) & = \mathbb{P}(A, v(i)=u) \\ 
    & = \sum_{t \leq n^2} \mathbb{P}(A|v(i)=u,T=t) \mathbb{P}(v(i)=u|T=t) \mathbb{P}(T=t)\\
    & \leq \sum_{t \leq n^2} \mathbb{1}_A(t) \frac{d_u^{(n)}}{n\lambda^{(n)}-t}\mathbb{P}(T=t) = \mathbb{E}\left(\mathbb{1}_A  \frac{d_u^{(n)}}{n\lambda^{(n)}-T}\right)
\end{align*}

In the forth line we have used that $\mathbb{P}(A|v(i)=u,T=t) = \mathbb{1}_A(t)$ because $A$ is $\sigma(T)$-measurable and that because the matching is done uniformly between all the unmatched half-edges at time $i$ we have that $\mathbb{P}(v(i)=u|T=t) = \frac{d_u(i)^{(n)}}{n\lambda^{(n)}-(i-1)} \leq \frac{d_u^{(n)}}{n\lambda^{(n)}-t}$ (with $d_u(i)^{(n)}$ as the number of unmatched half-edges of $u$ by the step $i$ of the exploration).

Now using that
\begin{equation*}
        N_1 = \sum_{i=1}^{T} \sum_{j>i}^{T}\mathbb{1}_{\{v(i)=v(j)\}}= \sum_{i=1}^{T} \sum_{j>i}^{T} \sum_{u\in V} \mathbb{1}_{\{v(i)=u\}} \mathbb{1}_{\{v(j)=u\}},
\end{equation*}

taking conditional expectation on $T$ over both sides and using this inequality, we get that

\begin{align*}
    \mathbb{E}\left( N_1 | T\right) & = \sum_{i=1}^{T} \sum_{j>i}^{T} \sum_{u\in V^{(n)}} \mathbb{P}(v(i)=u,v(j)=u|T)  \leq \sum_{i=1}^{T} \sum_{j>i}^{T} \sum_{u\in V^{(n)}} \left(\frac{d_u^{(n)}}{n\lambda^{(n)}-T}\right)^2\\
                                                           & \leq \frac{n\lambda^{(n)}}{(n\lambda^{(n)}-T)^2} \left( \sum_{i=1}^{T} \sum_{j>i}^{T} 1 \right) \left( \sum_{u\in V^{(n)}} d_u^{(n)} \frac{d_u^{(n)}}{n\lambda^{(n)}}\right)  = \frac{(\nu^{(n)} + 1) n\lambda^{(n)} T(T-1)}{2 (n\lambda^{(n)} -T)^2}\end{align*}

Where for the first inequality we used that for all $i < j$
\begin{align*}
    \mathbb{P}(v(i)=u,v(j)=u|T) & = \mathbb{P}(v(j)=u|v(i)=u,T) \mathbb{P}(v(i)=u|T) \\ & \leq \mathbb{P}(v(j)=k|T) \mathbb{P}(v(i)=k|T)
\end{align*}\end{proof}

By means of this lemma, if we denote by $B$ the number of \emph{bad vertices} that are in components that are not trees, we can prove that (under certain assumptions) it grows slower than any positive power of the graph size $n$.

\begin{proposition}\label{prop:2coreChico}
Let $G \sim \mbox{CM}_n(\overline{d})$ where the $(CA)$ holds towards a limiting degree distribution $(p_k)_{k\geq0}$ of mean $0 < \lambda < \infty$, criticality parameter $0 < \nu < 1$ and such that there exists $\gamma >0$ where $p_k = \mathcal{O}(e^{-\gamma k})$. Then, for every $0<\alpha<1$ we have that $B = \mathcal{O}_{\mathbb{P}}(n^{\alpha})$.
\end{proposition}
\begin{proof}
The proof will apply Lemma \ref{lemma:cotacond} and a coupling, presented by S. Janson et al. in \cite{janson2008largest}, which we describe now.
\\\\
Suppose we are in the step $i \leq \sqrt{n}$ (the exact power of $n$ is in fact irrelevant, it only needs to be $o(n)$ and of an order smaller than the components sizes) of the exploration process. Then the probability of finding a vertex of degree $k$ (excluding the vertices already explored) will be 
\begin{equation*}
    \frac{nk p_k^{(n)}}{ n \lambda^{(n)}-  \mathcal{O}(\sqrt{n})} = \frac{k p_k^{(n)}}{ \lambda^{(n)}  } ( 1 + \mathcal{O}(\sqrt{n}) )
\end{equation*}

Defining a random variable $X$ distributed according to

\begin{equation}\label{eq:treecoupling}
    \mathbb{P}(X \geq x ) = \mbox{min}\left( 1, \frac{\nu'}{\nu} \sum_{k \geq x} \frac{(k+1) p_{k+1}}{\lambda} \right) 
\end{equation}

For $\nu' = \nu + \epsilon'$ fixed and with $0 < \epsilon' < 1 - \nu$, then for every $n$ large enough $X$ stochastically dominates the step size of the random walk associated to the exploration process. As a consequence, we will have that if $\tilde T$ is the hitting time of $0$ of a random walk starting from $1$ and with step size $X$, $T \leq \tilde T$ a.s. whenever $\tilde T \leq \sqrt{n}$. The variable $\tilde T$ can also be thought of as the total progeny of a Galton-Watson process with progeny law given by i.i.d. copies of $X$. By summing over the expression in (\ref{eq:treecoupling}) we get that $\mathbb{E}(X) \leq \frac{\nu'}{\nu}\nu <1$, which means that the associated Galton-Watson process is subcritical.

Making use of this coupling, we can now prove the proposition. Taking expectation to $N_1$ and separating according to different values of $T$ we obtain that
\begin{equation*}
\mathbb{E}(N_1) = \mathbb{E}\left(N_1 \mathbb{1}_{\{T\leq n^\beta\}}\right) + \mathbb{E}\left(N_1 \mathbb{1}_{\{T > n^\beta\}} \right)
\end{equation*}
Where $0<\beta<1/2$ (its exact value will be fixed later on). By the upper bound on the conditional expectation in Lemma \ref{lemma:cotacond} we get that the first term is

\begin{equation*}
\mathbb{E}\left(N_1 \mathbb{1}_{\{T\leq n^\beta\}}\right) \leq \frac{(\nu^{(n)} + 1) n\lambda^{(n)} n^\beta(n^\beta-1)}{2 (n\lambda^{(n)} -n^\beta)^2} \mathbb{P}(T\leq n^\beta) \leq \frac{(\nu^{(n)} + 1) n\lambda^{(n)} n^\beta(n^\beta-1)}{2 (n\lambda^{(n)} -n^\beta)^2}
\end{equation*}

Because by hypothesis $\lambda^{(n)} \rightarrow \lambda$ and $\nu^{(n)} \rightarrow \nu$, the right hand is then $\mathcal{O}(n^{2\beta-1})$.

For the second term, we will again use that $N_1 \leq T$ a.s. to obtain that
\begin{equation*}
  \mathbb{E}\left(N_1 \mathbb{1}_{\{T > n^\beta\}} \right) \leq \mathbb{E}\left(T \mathbb{1}_{\{T > n^\beta\}} \right) \leq n \mathbb{P}(T > n^\beta)
\end{equation*}

Because of the coupling described above, we will also have that $\mathbb{P}(T > n^\beta) \leq \mathbb{P}(\tilde T > n^\beta)$ (because $n^\beta < \sqrt{n}$), where $\tilde T$ is the total progeny of a subcritical Galton-Watson process with progeny $X$, where $X$ is distributed according to \eqref{eq:treecoupling}. By hypothesis, for some $\gamma' > 0$, $\mathbb{P}(X = k) = \mathcal{O}(e^{-\gamma' k})$. Then, $X$ will have finite exponential moment and by Theorem 2.1 in \cite{nakayama2004finite}, so will $\tilde T$. Therefore, by Markov's inequality, we will have that for every $\theta > 0$

\begin{equation*}
    \mathbb{P}( \tilde T > n^\beta ) \leq \frac{\mathbb{E}( \tilde T^\theta )}{n^{\beta \theta}} = \mathcal{O}(n^{-\beta \theta})
\end{equation*}

Putting the three bounds together we obtain that (for $\theta$ large enough)

\begin{equation*}
    \mathbb{E}(N_1) = \mathcal{O}(n^{2\beta-1}) +\mathcal{O}(n^{1-\beta\theta}) = \mathcal{O}(n^{2\beta-1}) 
\end{equation*}

Now, fix $0< \alpha < 1$. Using this and Markov's inequality, we obtain that

\begin{equation*}
\mathbb{P}\left(N > n^{\alpha/2} \right) \leq \frac{\mathbb{E}(N)}{n^{\alpha/2}} \leq \frac{\mathbb{E}(\sum_{u\in V} N_u)}{n^{\alpha/2}} = \frac{n \mathbb{E}(N_1)}{n^{\alpha/2}} = \mathcal{O}(n^{-\delta_1})
\end{equation*}

Where $\delta_1 := \alpha/2 - 2 \beta$. By taking $\beta < \alpha / 4$ we have that $\delta_1 > 0$.

On the other hand, if we call $C_{max}$ the largest component of the graph, by Theorem 1.1 in \cite{janson2008largest}\footnote{Here we use that the tails of the degree distribution are exponentially thin and therefore $\mathcal{O}(k^{-\gamma''})$ for every $\gamma'' > 0$.} we also have that 
\begin{equation*}
 \mathbb{P}(|C_{max}| > n^{\alpha/2}) = \mathcal{O}(n^{-\delta_2})   
\end{equation*}
for some $\delta_2 > 0$. Then, taking $\delta := \min(\delta_1,\delta_2)>0$, we have that
\begin{equation*}
 \mathbb{P}( B > n^\alpha ) \leq \mathbb{P}(N |C_{max}| > n^\alpha) \leq \mathbb{P}(N > n^{\alpha/2})+\mathbb{P}(|C_{max}| > n^{\alpha/2}) = \mathcal{O}(n^{-\delta})
\end{equation*} \end{proof}

This last proposition shows that the number of vertices in components that are not trees is (for every $\alpha>0$) $\mathcal{O}_{\mathbb{P}}(n^\alpha)$. Now, define $T_G$ as a graph formed by spanning trees of each of the components of $G$ (it is not important which specific ones are chosen). We call $W_{DG}(G)$ and $W_{DG}(T_G)$ the selection sequences defined by the degree-greedy algorithm run in $G$ and $T_G$ respectively. Observe that the components that are trees look exactly the same in $G$ and $T_G$.

We can then couple the realisations of the degree-greedy algorithm in $G$ and $T_G$ to make them coincide in these components in the following way: 

\begin{itemize}
    \item Call the connected components of $G$ as $C_1$, $C_2$,..., $C_l$ and the ones of $T_G$ as $C_1'$, $C_2'$,..., $C_l'$.
    \item Run a degree-greedy algorithm in each of the components. This generates the selection sequences $W_1$, $W_2$,..., $W_l$ for the components of $G$ and $W_1'$, $W_2'$,..., $W_l'$ for the ones of $T_G$. If for some $j \leq l$ the component $C_j$ is a tree, then $C_j = C_j'$ and the respective runs of the degree-greedy algorithm can be trivially coupled to give $W_j = W_j'$. Couple in this manner all the selection sequences of all the components that are trees.

    \item Now, construct $W_{DG}(G)$ inductively as follows: in each step $i \geq 1$, count the number of minimum degree vertices in each component $j \leq l$ of $G_{i-1}$ and call this number $d_j^{(i)}$. Select component $j \leq l$ with probability $d_j^{(i)}/\sum_{k=1}^m d_k^{(i)}$. Set $w_i$ (the $i$-th vertex of $W_{DG}(G)$) as the first vertex of $W_j$ not already in $\{w_1,...,w_{i-1}\}$.
    \item Finally, construct $W_{DG}(T_G)$ in an analogous way but using selection sequences $W_1'$, $W_2'$,..., $W_l'$.
\end{itemize}

It is straightforward that this construction corresponds to the degree-greedy sequential exploration in each of both graphs.

The construction was made in such a way that for every component that is a tree, the same vertices end up in $W_{DG}(G)$ as in $W_{DG}(T_G)$. This means by Proposition \ref{prop:2coreChico} that, at most, $|W_{DG}(G)|$ and $|W_{DG}(T_G)|$ will differ in size by $\mathcal{O}_\mathbb{P}(n^\alpha)$ (the number of vertices in components that are not trees). Besides, $T_G$ is a collection of trees and therefore, by Corollary \ref{coro:arbol}, the degree-greedy ran on it will define a selection sequence with the property $T_1$, and then by Lemma \ref{lemma:optimal} $|W_{DG}(T_G)| = \alpha(T_G)$. But all the edges in $T_G$ are also present in $G$ and so $T_G$ will have a bigger maximum independent set than $G$\footnote{This is because the independence number is monotonically decreasing when adding edges to a graph.}. This implies that 

\begin{equation*}
 |W_{DG}(T_G)| = |W_{DG}(G)| + \mathcal{O}_\mathbb{P}(n^\alpha) = \alpha(T_G) \geq \alpha(G)   
\end{equation*}

which in term implies that (because $|W_{DG}(G)|= \sigma_{DG}(G)$)

\begin{equation*}
 |\alpha(G)-|W_{DG}(G)|| \leq \mathcal{O}_\mathbb{P}(n^\alpha)
\end{equation*}
proving the proposition.

We now give a proof for the case in which $G$ is a supercritical graph. Call $W_{DG}$ the selection sequence defined by the degree-greedy algorithm ran on $G$. We want to show that $|W_{DG}|= \alpha(G) + \mathcal{O}_\mathbb{P}(n^\alpha)$. W.h.p. we have that this sequence selects vertices of degree $1$ or $0$ at least until the remaining graph is subcritical. Suppose this is so, then there exists some value $k_0 \geq 1$ s.t. for every $k \geq k_0$ the remaining graph $G_k$ (see definition \ref{def:remaining}) is subcritical and for every $l \leq k_0$ the vertex $W_{DG}(l)$ has degree either $1$ or $0$ in $G_l$.

The idea is to define a selection sequence $\tilde W$ that is similar to $W_{DG}$, has roughly the same size and for which $|\tilde W| \geq \alpha(G)$. Calling $T_{G_{k_0}}$ the graph formed by the spanning trees of $G_{k_0}$, we define the graph $\tilde G$ as a copy of $G$ in which the subgraph $G_{k_0}$ has been replaced by $T_{G_{k_0}}$. We can then define a selection sequence $\tilde W$ for $\tilde G$ that coincides with $W_{DG}$ until step $k_0$. For $k \geq k_0$, because $G_{k_0}$ is subcritical and the remaining graph of $\tilde G$ is a collection of spanning trees of $G_{k_0}$ (that is, $T_{G_{k_0}}$), we can make $\tilde W$ to have the property $T_1$ and to differ in at most $\mathcal{O}_\mathbb{P}(n^\alpha)$ vertices from $W_{DG}$ in exactly the same way as in the subcritical case. Because $\tilde W$ has the property $T_1$, using Lemma \ref{lemma:optimal}, $|\tilde W| = \alpha(\tilde G)$. Furthermore, because all the edges present in $\tilde G$ are present in $G$ we have that $\alpha(\tilde G) \geq \alpha(G)$. Then, 

\begin{equation*}
     |\tilde W| = |W_{DG}| + \mathcal{O}_\mathbb{P}(n^\alpha) = \alpha(\tilde G) \geq \alpha(G)
\end{equation*}

which means that 

\begin{equation*}
 |\alpha(G)-\sigma_{DG}(G)| \leq \mathcal{O}_\mathbb{P}(n^\alpha)   
\end{equation*}

as we wanted to show.


\subsection{Hydrodynamic limit results}\label{sec:fluidres}

In this section we present the results we will use in the remaining of the paper to show the convergence of processes and stopping times towards deterministic limits. These results are not necessarily presented in full generality but rather in the most convenient form for the applications we intend.

\begin{definition}
 Given a continuous time Markov jump process $A_t\in D[0,\infty)$ we will define its associated \textbf{Dynkin's martingales}\footnote{Which are martingales because of Dynkin's formula.} by
    
    \begin{equation*}
        M_t : = A_t - A_0 - \int_0^t \delta[A_s] ds
    \end{equation*}
    
    \noindent
    where $\delta(\cdot)$ is the drift of $A_t$. 
\end{definition}

The following lemma will be our main tool to prove convergence of stochastic processes towards solutions of differential equations. It is an abstraction of the reasoning behind the proof of the limits in the main theorem of \cite{brightwell2016greedy}.

\begin{lemma}\label{lema:limitefluido}
    Let $(A_t^{(n)}(1),A_t^{(n)}(2),...) \in D[0,\infty)^\mathbb{N}$ be a sequence of countable continuous time Markov jump processes where (for each $k\in\mathbb{N}$) $A_t^{(n)}(k)$ has drift $\delta_k(A_t^{(n)}(1),A_t^{(n)}(2),...)$. If (for every $k\in\mathbb{N}$):
    
    \begin{enumerate}[(i)]
        \item $\delta_k(A_t^{(n)}(1),A_t^{(n)}(2),...)/n = \sum_{i\geq1}^{i_k} \alpha_i A_t^{(n)}(i)/n$, where $i_k\in \mathbb{N}$
        \item $\delta_k(A_t^{(n)}(1),A_t^{(n)}(2),...)/n$ are uniformly bounded
        \item $A_0^{(n)}(k)/n \xrightarrow{n\rightarrow\infty} a_0(k)$ (for some constant $a_0(k)$)
        \item the associated Dynkin's martingales $M_t^{(n)}(k)$ have quadratic variation of order $o_{\mathbb{P}}(n^2)$
    \end{enumerate}
    
    then, if the system of integral equations defined by
    
    \begin{equation*}
        a_t(1) = a_0(1) + \int_0^t \delta_1(a_s(1),a_s(2),...) ds
    \end{equation*}
    
    \begin{equation*}
        a_t(2) = a_0(2) + \int_0^t \delta_2(a_s(1),a_s(2),...) ds
    \end{equation*}
    
    \begin{equation*}
        . \ . \ .
    \end{equation*}
    
    has a unique solution, the processes $A_t^{(n)}(1)$, $A_t^{(n)}(2)$,... converge in probability towards the continuous functions $a_t(1)$, $a_t(2)$ that are solution to the system.
    
\end{lemma}

\begin{proof}
Dividing by $n$ Dynkin's formula we have that

\begin{equation*}
    \frac{A_t^{(n)}(1)}{n} = \frac{A_0^{(n)}(1)}{n} + \int_0^t \frac{ \delta_1^{(n)}[A_s^{(n)}(1),A_s^{(n)}(2),...]}{n}ds + \frac{M_t^{(n)}(1)}{n} 
\end{equation*}

\begin{equation*}
    \frac{A_t^{(n)}(2)}{n} = \frac{A_0^{(n)}(2)}{n} + \int_0^t \frac{ \delta_2^{(n)}[A_s^{(n)}(1),A_s^{(n)}(2),...]}{n}ds + \frac{M_t^{(n)}(2)}{n} 
\end{equation*}

\begin{equation*}
    . \ . \ .
\end{equation*}

Because the associated Dynkin's martingales have quadratic variation of order $o_\mathbb{P}(n^2)$ by Doob's inequality we have that, for each $k\in\mathbb{N}$, $\sup_{s\leq t}|M_t^{(n)}|/n$ converges uniformly in distribution in $C[0,\infty)$ towards 0. Since the $\delta_k^{(n)}(\cdot)$ are uniformly bounded, for every $k\in\mathbb{N}$ and $T>0$, the sequences of processes $(A_t^{(n)}(k)-M_t^{(n)})/n$ are uniformly Lipschitz and uniformly bounded in $[0,T]$. Then, by Arzela-Ascoli's Theorem, for every $T >0$ these families of processes are tight in $C[0,T]$. This implies that for every subsequence there exists subsubsequences such that 

\begin{equation*}
    \frac{A_t^{(n)}(1)-M^{(n)}_t(1)}{n} \rightarrow a_t(1)
\end{equation*}

\begin{equation*}
    \frac{A_t^{(n)}(2)-M^{(n)}_t(2)}{n} \rightarrow a_t(2)
\end{equation*}

\begin{equation*}
    . \ . \ .
\end{equation*}

uniformly in distribution in compact sets, for some continuous functions $a_t(1)$, $a_t(2)$, etc. Since the processes are countable, we may take a common subsubsequence where all this convergences hold. Furthermore, by Skorohod's Representation Theorem, there exists an (abstract) probability space s.t. all these limits and the convergence (for every $k\in\mathbb{N}$) of $\sup_t |M^{(n)}_t(k)|/n$ towards $0$ hold uniformly and almost surely in compact sets. Then, for this subsubsequence

\begin{equation*}
    \frac{A_t^{(n)}(1)}{n} \xrightarrow{a.s.} a_t(1)
\end{equation*}

\begin{equation*}
    \frac{A_t^{(n)}(2)}{n} \xrightarrow{a.s.} a_t(2)
\end{equation*}

\begin{equation*}
    . \ . \ .
\end{equation*}

By hypothesis $\delta_k^{(n)}[A_s^{(n)}(1),A_s^{(n)}(2),...]/n = \sum_{\geq i}^{i_k} \alpha_i^{(k)} A_t^{(n)}(k)/n$, then $\lim_n \delta_k(A_t^{(n)}(1),A_t^{(n)}(2),...)/n = \delta_k(a_t(1),a_t(2),...)$. And because all the drifts are uniformly bounded, taking limit over this subsubsequence and using that $A_t^{(n)}(k)/n\rightarrow a_0(k)$ and dominated convergence yields

\begin{equation*}
    a_t(1) = a_0(1) + \int_0^t \delta_1(a_s(1),a_s(2),...) ds
\end{equation*}

\begin{equation*}
    a_t(2) = a_0(2) + \int_0^t \delta_2(a_s(1),a_s(2),...) ds
\end{equation*}

\begin{equation*}
    . \ . \ .
\end{equation*}

The convergence towards the solution of this system of integral equations is well defined as, by hypothesis, it has a unique solution. We now need to prove that this convergence is not only in this subsubsequence but rather in the whole original sequence. For this, note that since the limits are continuous and deterministic, this convergence is equivalent to convergence in distribution in the Skorohod topology on $D[0,\infty)$. But because every subsequence has a subsubsequence that converges to the same limit (because by hypothesis it is unique), the original sequence converges in distribution to it. Moreover, as the limit is deterministic, the convergence can equivalently be taken to be in probability.\end{proof}

In the following lemma, we will establish convergence criteria for the stopping times of sequences of decreasing processes that converge towards an hydrodynamic limit. 

\begin{lemma}\label{lema:convtiempo}
    Let $(A_t^{(n)}(1),A_t^{(n)}(2),...) \in D[0,\infty)^\mathbb{N}$ be a sequence of countable continuous time Markov jump processes with $A^{(n)}_t(1)$ a decreasing process of transition matrix $Q^{(n)}_{ij}$ and define the stopping times $T^{(n)}:=\inf \{ t\geq0 : A_t^{(n)}(1) = 0 \}$ and the deterministic time $T:=\inf\{s\geq0: a_s(1)=0\}$. Under the same hypothesis of previous lemma and further assuming that:
    
    \begin{enumerate}[(i)]
        \item For every $t \leq T^{(n)}$ and (if $A_t^{(n)}(1) = i \geq 0$), $\sum_{j\leq i}Q^{(n)}_{ij}(A_t^{(n)}(2),...) \geq C^{(n)}n$ with $C^{(n)} \xrightarrow{n\rightarrow\infty} C > 0$
        \item The function $a_t(1)$ is continuously differentiable with $\dot a_t(1) \leq - C'$ (for some $C' >0$ and $t\leq T$)
    \end{enumerate}
    
    \noindent
    then, $T^{(n)} \xrightarrow{\mathbb{P}}T$.
\end{lemma}

\begin{proof}
    Given $\delta > 0$, we want to show that the probability of $\{|T-T^{(n)}|\geq\delta\}$ goes to 0. For this, suppose that $T > T^{(n)}$. Now, suppose that the event $\{\sup_{t\leq T}|A_t^{(n)}(1)/n-a_t(1)|\leq \delta C'\}$ holds. Then, $a_{T^{(n)}}(1)$ is at most $\delta/C'$. By hypothesis we have that for every $t\leq T$ the derivative of $a_t(1)$ is less than $-C'$, then $a_{T^{(n)}+\delta}(1) \leq a_{T^{(n)}}(1) - C' \delta \leq 0$. Therefore, $\{T-T^{(n)} \geq \delta\} \subseteq \{\sup_{t\leq T}|A_t^{(n)}(1)/n-a_t(1)| \geq \delta C'\}$. And because $A_t^{(n)}(1)/n\xrightarrow{\mathbb{P}}a_t(1)$, the probability of this last event tends to 0.
    
    On the other hand, now suppose that $T^{(n)} > T$. If the event $\{\sup_{t\leq T}|A_t^{(n)}(1)/n-a_t(1)|\leq C \delta / 4\}$ holds, the value of $A_{T}^{(n)}/n$ will be at most $C \delta / 4$. Because by hypothesis $A_t^{(n)}(1)$ is decreasing and if $A_t^{(n)}(1)=i\geq0$ the process has transitions to lower states with rate at least $n C^{(n)} = n C + o(n)$. If we define $(Z_t)_{t\geq T}$ to be a pure death process with initial value $Z_T = C \delta / 4$ and death rate $n C / 2$, we can then couple (for $n$ larger than certain $n_0\geq0$) the process $A_t^{(n)}$ to be larger than $Z_t$ for $t \geq T$. Then, for $n \geq n_0$, $\{T^{(n)} - T \geq \delta\} \subseteq \{ Z_{T+\delta} \geq 0 \}$. But, defining $X\sim \mbox{Pois}(n C \delta / 2)$, the probability of this last event is equal to $\mathbb{P}(X \leq n C \delta / 4)$ and by Chebychev's inequality we have that
    
    \begin{equation*}
        \mathbb{P}( X \leq n C \delta / 4 ) = \mathbb{P}( nC\delta / 2 - X \geq n C \delta /4) \leq \mathbb{P}( | X - nC\delta / 2| \geq n C \delta /4) \leq \frac{n C \delta/2}{n^2 C^2 \delta^2 / 16} = \frac{8}{n C \delta}
    \end{equation*}
    
    Summarising,
    
    \begin{equation*}
	\begin{array}{cl}
        	\mathbb{P}( |T^{(n)}-T| \geq \delta ) & = \mathbb{P}( T  - T^{(n)} > \delta) \\
        	& \ \ + \mathbb{P}( \sup_{t\leq T}|A_t^{(n)}(1)/n-a_t(1)|\leq C \delta / 4, T^{(n)} - T > \delta ) \\
		& \ \  +  \mathbb{P}( \sup_{t\leq T}|A_t^{(n)}(1)/n-a_t(1)|\geq C \delta / 4, T^{(n)} - T > \delta ) \\
		
		& \leq  \mathbb{P}(\sup_{t\leq T}|A_t^{(n)}(1)/n-a_t(1)|\geq \delta C' ) \\
		& \ \ + \mathbb{P}( \sup_{t\leq T}|A_t^{(n)}(1)/n-a_t(1)|\leq C \delta / 4, T^{(n)} - T > \delta ) \\
		& \ \ +  \mathbb{P}( \sup_{t\leq T}|A_t^{(n)}(1)/n-a_t(1)|\geq C \delta / 4 ) \\ 
		
		& \leq  \mathbb{P}(\sup_{t\leq T}|A_t^{(n)}(1)/n-a_t(1)|\geq \delta C' ) + \mathbb{P}( X \leq n C \delta / 4 ) \\
		& \ \ +  \mathbb{P}( \sup_{t\leq T}|A_t^{(n)}(1)/n-a_t(1)|\geq C \delta / 4 ) \xrightarrow{n\rightarrow\infty} 0 
        \end{array}
\end{equation*}
    
   where the last term $\mathbb{P}( \sup_{t\leq T}|A_t^{(n)}(1)/n-a_t(1)|\geq C \delta / 4 )$ goes to $0$ because we are under the hypothesis of the previous lemma.
\end{proof}

We now need to establish the convergence of the coordinates of countable Markov jump processes at specific stopping times towards corresponding values of its hydrodynamic limit. For this, we will use the following corollary. 

\begin{corollary}\label{coro:convcoord}
    Let $(A_t^{(n)}(1),A_t^{(n)}(2),...) \in D[0,\infty)^\mathbb{N}$ be as in the previous lemma. Then, for every $k\in\mathbb{N}$ we will have that $A^{(n)}_{T^{(n)}}(k)/n \xrightarrow{\mathbb{P}}a_T(k)$. 
\end{corollary}

\begin{proof}
We have that

\begin{equation*}
 |A_{T^{(n)}}(k)/n-a_T(k)| \leq |A_{T^{(n)}}(k)/n-a_{T^{(n)}}|+|a_{T^{(n)}}-a_T(k)|   
\end{equation*}

where the first term in the r.h.s. goes w.h.p. to 0 because by Lemma \ref{lema:limitefluido} $A_t^{(n)}(k)/n\xrightarrow{\mathbb{P}}a_t(k)$ uniformly on compact sets (and since $T^{(n)}\xrightarrow{\mathbb{P}}T$, the $T^{(n)}$ are w.h.p. uniformly bounded); and the second one because $T^{(n)}\xrightarrow{\mathbb{P}}T$ by Lemma \ref{lema:convtiempo}  and $a_t(k)$ is continuous. \end{proof}

\subsection{Proof of Theorem \ref{thm:oneapplication}}

For the analysis of the effect of $M_1^{(n)}(\cdot)$ we will break the degree-greedy dynamics in two. In the \emph{first phase}, we will only connect the vertices of degree $1$ and we will accumulate the number of free half-edges stemming from the blocked vertices in a variable $B^{(n)}_t$. We will do this until we have explored $n p_1^{(n)}$ vertices of degree $1$. While in the \emph{second phase}, we will take the final value of $B_t^{(n)}$ from phase $1$ and we will sequentially match each of this half-edges and remove the edges formed from the graph.

Because the Configuration Model is not sensitive to the order in which the matching of half-edges is done, this will not affect the final degree distribution obtained (which will be the same as the one obtained applying the map $M_1^{(n)}(\cdot)$) but will nevertheless make the analysis of the resulting limit easier. As we will see, for the proof of Theorem \ref{thm:oneapplication} it will only be enough to analyse the first of these two phases.

The proof will consist of three parts. First, we give a stochastic description of a Markov process that gives the evolution of the first phase. Then, we establish the concentration of the asymptotic values of the number of unmatched half-edges, the number of unmatched half-edges stemming from blocked vertices and the normalised degree measure of the vertices not connected to any degree $1$. Using this convergences we prove that the resulting graph after phase 1 can be regarded as a Configuration Model with known limiting distribution. Finally, these limits combined with an observation from percolation for Configuration Models, allow us to establish the criteria given by Theorem \ref{thm:oneapplication}.

\noindent
\emph{(i) Stochastic description of the first phase:} the stochastic process used to model the first phase of the matching of the degree 1 vertices will be similar in spirit to the one used in \cite{brightwell2016greedy} to study the greedy algorithm. Here, we also keep track of the number of unpaired blocked half-edges in the random variable $B_t^{(n)}$. The other variables used to describe the process will be: the number $U_t^{(n)}$ of unpaired half-edges, the number $A_t^{(n)}$ of remaining degree 1 vertices to match and (for $k \in \mathbb{N}$) the number $\mu_t^{(n)}(k)$ of unexplored degree $k$ vertices.

Then, at each time $t \geq 0$ the state of the process will be described by the infinite dimensional vector $(U_t^{(n)},A_t^{(n)},B_t^{(n)},\mu_t^{(n)}(2),\mu_t^{(n)}(3),...)$. The process in question evolves as follows: at each time $t \geq 0$ every degree 1 vertex in $A_t^{(n)}$ will have an exponential clock with rate $1$. When one of the clocks of some of these vertices rings, the vertex is removed from $A_t^{(n)}$ and its edge is uniformly paired to another unpaired edge. The state of the vertex with the half-edge selected for the pairing is declared \emph{blocked} and its unmatched half-edges are added to $B_t^{(n)}$. The process goes on until there are no more degree 1 vertices to be paired. Because it is defined by well-behaved transition rates, it is straightforward that the resulting process is Markovian. The stopping time in which the process finishes will be given by $T_1^{(n)} := \inf \{ t\geq0 : A_t^{(n)} = 0 \}$ and will be a.s. finite.
\\\\
\emph{(ii) Convergence of the final values of the coordinates:}

To analyse the convergence of the final value of the number of unmatched half-edges, we will represent with $\{e_i \not\leftrightarrow (1)\}$ the event that the half-edge $e_i$ is not matched nor stemming from a degree $1$ vertex. We will then define the r.v. $Y^{(n)} := \sum_{i=1}^{n \lambda^{(n)}} \mathbb{1}_{\{e_i \not\leftrightarrow (1)\}} = U_{T_1^{(n)}}^{(n)}$ that gives the number of unmatched half-edges at the end of the first phase. The convergence will be proved by showing that this variable concentrates around its mean. For this, we will then first compute its corresponding value:

\begin{equation*}
    \mathbb{E}\left(Y^{(n)}\right) = \mathbb{E}\left( \sum_{i=1}^{n \lambda^{(n)}} \mathbb{1}_{\{e_i \not\leftrightarrow (1)\}} \right)
\end{equation*}

Because half-edges are interchangeable, this is equal to $n \lambda^{(n)} \mathbb{P}(e_1 \not\leftrightarrow (1))$. The probability that the event $\{ e_1 \not\leftrightarrow (1) \}$ holds may be seen to be given by $\frac{n \lambda^{(n)} - n p_1^{(n)} - 1}{n \lambda^{(n)} - 1} \frac{n \lambda^{(n)} - n p_1^{(n)} - 2}{n \lambda^{(n)} - 2} = (1-p_1/\lambda)^2 + o(1)$.

We then have that 

\begin{equation}\label{eq:convmediaY}
    \mathbb{E}\left(Y^{(n)}\right) = n \lambda^{(n)} (1-p_1/\lambda)^2 + o(n)
\end{equation}

To show the concentration, we will now bound the variance of this variable:

\begin{equation}\label{eq:varYi1}
    \mbox{Var}\left( Y^{(n)} \right) = \mathbb{E}\left[ \left(Y^{(n)}\right)^2\right] - \mathbb{E}\left(Y^{(n)}\right)^2
\end{equation}

Where the first term in \eqref{eq:varYi1} will be given by

\begin{equation*}
    \mathbb{E}\left( \sum_{i\neq j}^{n \lambda^{(n)}} \mathbb{1}_{\{e_j \not\leftrightarrow (1)\}} \mathbb{1}_{\{e_j \not\leftrightarrow (1)\}} \right) +  \mathbb{E}\left( Y^{(n)} \right)
\end{equation*}

Then, defining $A := \mathbb{E}\left( \sum_{i\neq j}^{n \lambda^{(n)}} \mathbb{1}_{\{e_j \not\leftrightarrow (1)\}} \mathbb{1}_{\{e_j \not\leftrightarrow (1)\}} \right)$, $B:=\mathbb{E}\left( Y^{(n)} \right)$ and $C:= \mathbb{E}\left(Y^{(n)}\right)^2$, we have that $\mbox{Var}\left( Y^{(n)} \right) = A + B - C$. 

The term $A$ can be bounded above by

\begin{equation*}\label{eq:cotasupvarY}
    n \lambda^{(n)}(n \lambda^{(n)} -1) \mathbb{P}(e_1 \not\leftrightarrow e_2, e_1 \not\leftrightarrow (1), e_2 \not\leftrightarrow (1))
\end{equation*}

With $\{e_1 \not\leftrightarrow e_2\}$ representing the event that $e_1$ is not matched to $e_2$. This probability is given by

\begin{equation*}
    \mathbb{P}(e_1 \not\leftrightarrow (1),e_2 \not\leftrightarrow (1)|e_1 \not\leftrightarrow e_2)\mathbb{P}(e_1 \not\leftrightarrow e_2).
\end{equation*}

It can be easily seen that $\mathbb{P}(e_1 \not\leftrightarrow (1),e_2 \not\leftrightarrow (1)|e_1 \not\leftrightarrow e_2) = \prod_{k=1}^4 \left( 1 - \frac{n p_1^{(n)}}{n \lambda^{(n)} - k} \right)$, and that $\mathbb{P}(e_1 \not\leftrightarrow e_2) = 1 + o(1)$.

We then get that $A \leq n \lambda (n \lambda-1) (1-p_1/\lambda)^{4} + o(n^2)$. Also, by \eqref{eq:convmediaY}, $B = n \lambda (1-p_1/\lambda)^2 + o(n)$ and $C = n^2 \lambda^{2} (1-p_1/\lambda)^{4} +o(n^2)$. This shows that the variance of $Y^{(n)}$ is $o(n^2)$. 

Using Chebychev's inequality we then get that (for every $\epsilon>0$)

\begin{equation}
    \mathbb{P}\left(|Y^{(n)}-\mathbb{E}(Y^{(n)})| > \epsilon n\right) \leq \frac{\mbox{Var}\left(Y^{(n)}\right)}{\epsilon^2 n^2} \rightarrow 0
\end{equation}

Which proves that $U_{T_1^{(n)}}^{(n)}/n$ converges in probability towards $u_1:=(1-p_1/\lambda)^2 \lambda=Q^2\lambda$.

The computation for the convergence of the final value of the degree measure will be completely analogous. We will represent with $\{v_j \not\leftrightarrow (1)\}$ the event that the vertex $v_j$ is not connected to any degree $1$ vertex at time $T_1^{(n)}$. Then, (for $i\geq2$) the random variable $Z_i^{(n)} := \sum_{j=1}^{n} \mathbb{1}_{\{v_j \not\leftrightarrow (1)\}} \mathbb{1}_{\{d_{v_j}=i\}} = \mu_{T_1^{(n)}}^{(n)}(i)$ gives the number of degree $i$ vertices not connected to any degree $1$ vertex at the end of the first phase. In other words, $Z_i^{(n)}$ gives the number of vertices of degree $i$ that remain after the first phase. We will first prove that these variables converge in probability.

For this, we compute their corresponding mean values:

\begin{equation*}
    \mathbb{E}\left(Z_i^{(n)}\right) = \mathbb{E}\left( \sum_{j=1}^{n} \mathbb{1}_{\{v_j \not\leftrightarrow (1)\}} \mathbb{1}_{\{d_{v_j}=i\}} \right)
\end{equation*}

Which, because vertices are interchangeable, is equal to $n \mathbb{P}(v_1 \not\leftrightarrow (1),d_{v_1}=i)$. This last probability is easy to compute and gives $p_i^{(n)} \prod_{l=1}^i \left(1 - \frac{n p_1^{(n)}}{n \sum_{i\geq1}i p_i^{(n)}-(2l-1)} \right)$, which converges as $n\rightarrow\infty$ to $(1-p_1/\lambda)^i p_i$.

We then have that 

\begin{equation}\label{eq:convmedia}
    \mathbb{E}\left(Z_i^{(n)}\right) = n (1-p_1/\lambda)^i p_i + o(n)
\end{equation}

We will now bound the variance of these variables:

\begin{equation}\label{eq:varZi1}
    \mbox{Var}\left( Z_i^{(n)} \right) = \mathbb{E}\left[ \left(Z_i^{(n)}\right)^2\right] - \mathbb{E}\left(Z_i^{(n)}\right)^2
\end{equation}

Where the first term in \eqref{eq:varZi1} will be given by

\begin{equation*}
    \mathbb{E}\left( \sum_{j\neq k}^n \mathbb{1}_{\{v_j \not\leftrightarrow (1)\}} \mathbb{1}_{\{v_k \not\leftrightarrow (1)\}} \mathbb{1}_{\{d_{v_j}=d_{v_k}=i\}} \right) +  \mathbb{E}\left( Z_i^{(n)} \right)
\end{equation*}

Then, defining $A' := \mathbb{E}\left( \sum_{j\neq k}^n \mathbb{1}_{\{v_j \not\leftrightarrow (1)\}} \mathbb{1}_{\{v_k \not\leftrightarrow (1)\}} \mathbb{1}_{\{d_{v_j}=d_{v_k}=i\}} \right)$, $B':=\mathbb{E}\left( Z_i^{(n)} \right)$ and $C':= \mathbb{E}\left(Z_i^{(n)}\right)^2$, we have that $\mbox{Var}\left( Z_i^{(n)} \right) = A' + B' - C'$. 

Now, the term $A'$ can be bounded above by

\begin{equation*}\label{eq:cotasupvar}
    n(n-1) \mathbb{P}(v_1 \not\leftrightarrow v_2,v_1 \not\leftrightarrow (1),v_2 \not\leftrightarrow (1),d_{v_1}=d_{v_2}=i)
\end{equation*}

Where $\{v_1 \not\leftrightarrow v_2\}$ represents the event that $v_1$ is not connected to $v_2$ (where $v_1$ and $v_2$ are two distinct uniform vertices). This probability is equal to

\begin{equation*}
    \mathbb{P}(v_1 \not\leftrightarrow (1),v_2 \not\leftrightarrow (1)|v_1 \not\leftrightarrow v_2,d_{v_1}=d_{v_2}=i)\mathbb{P}(v_1 \not\leftrightarrow v_2|d_{v_1}=d_{v_2}=i)\mathbb{P}(d_{v_1}=d_{v_2}=i)
\end{equation*}

Furthermore, simple computations show that

\begin{equation*}
\begin{array}{cl}
    \mathbb{P}(v_1 \not\leftrightarrow (1),v_2 \not\leftrightarrow (1)|v_1 \not\leftrightarrow v_2,d_{v_1}=d_{v_2}=i) = \prod_{l=1}^i & \left(1 - \frac{n p_1^{(n)}}{n \sum_{i\geq1}i p_i^{(n)}-(i+2l-1)} \right) \\
    & \times \left(1 - \frac{n p_1^{(n)}}{n \sum_{i\geq1}i p_i^{(n)}-(2i+2l-1)} \right),
\end{array}
\end{equation*}

that $\mathbb{P}(v_1 \not\leftrightarrow v_2|d_{v_1}=d_{v_2}=i) = 1 + o(1)$ and that $\mathbb{P}(d_{v_1}=d_{v_2}=i) = p_i^{(n)2} + o(1)$.

Putting all this together shows that $A' \leq n(n-1) (1-p_1/\lambda)^{2i}p_i^2 + o(n^2)$. Finally, by \eqref{eq:convmedia}, $B' = n (1-p_1/\lambda)^i p_i + o(n)$ and $C' = n^2 (1-p_1/\lambda)^{2i} p_i^2 +o(n^2)$. Which proves that the variance of $Z_i^{(n)}$ is $o(n^2)$. 

By Chebychev's inequality this implies that

\begin{equation}
    \mathbb{P}\left(|Z_i^{(n)}-\mathbb{E}(Z_i^{(n)})| > \epsilon n\right) \leq \frac{\mbox{Var}\left(Z_i^{(n)}\right)}{\epsilon^2 n^2} \rightarrow 0
\end{equation}

Which in turn means that for any initial asymptotic degree distribution that is bounded, all the $Z_i^{(n)}/n$ will converge jointly in probability to $(1-p_1/\lambda)^i p_i$. But because $\left(Z_i^{(n)}\right)_{n\geq0}$ is bounded by $(p_i^{(n)})_{n\geq0}$ and this last sequence is eventually uniformly summable, then

\begin{equation}
 \left(Z_2^{(n)}/n,...,Z_i^{(n)}/n,...\right) \xrightarrow{\mathbb{P}} \left(\mu_1(2),...,\mu_1(i),...\right) := \left( Q^2 p_2,...,Q^ip_i,...\right)    
\end{equation}

Furtheremore, because we have that (for all $t\geq0$) $U_t^{(n)} = A_t^{(n)} + B_t^{(n)} + \sum_{i=2}^\infty i \mu_t^{(n)}(i)$ a.s., the normalised number $B_{T_1^{(n)}}/n$ of unmatched half-edges stemming from blocked vertices at the end of the first phase converges in probability towards $b_1 := u_1 - \sum_{i=2}^\infty i \mu_1(i)$.

Finally, as $\sum_{k\geq1} k^2 p_k^{(n)}$ converges towards $\sum_{k\geq1} k^2 p_k$, then these sums are uniformly summable. And because, for every $k\geq2$, we have that $\mu_{T_1^{(n)}}^{(n)}(k)/n \leq p_k^{(n)}$, the sums $\sum_{k\geq2}k^2 \mu^{(n)}_{T_1^{(n)}}/n$ will also be uniformly summable. This implies that

\begin{equation*}
    \sum_{k\geq2} k^2 p_k^{(n)}\mu^{(n)}_{T_1^{(n)}}/n\xrightarrow{\mathbb{P}} \sum_{k\geq2} k^2 \mu_1(k)
\end{equation*}

We then recover the (CA) for the degree distribution of the remaining graph after the first phase. This means that, if we regard the unexplored blocked half-edges as degree $1$ vertices, the graph obtained when the first phase is finished can be treated as a Configuration Model\footnote{For this we also need the number of remaining vertices to tend to infinity when $n\rightarrow\infty$, but this can be easily checked to be the case if $p_1<1$.} with limiting degree distribution $\tilde p_1 = b_1/K = (u_1-\sum_{i\geq2}\mu_1(i))/K$, and (for $k\geq2$) $\tilde p_k  = \mu_1(k)/K$ (where $K$ is just a normalisation constant).
\\\\
\emph{(iii) Criteria for subcriticality:} here we will establish under which circumstances the graph obtained after applying the map $M_1^{(n)}(\cdot)$ is w.h.p. subcritical. Note that, to do so, one in principle has to analyse what happens to the degree distribution during the second phase of the dynamics and then determine if the distribution obtained is subcritical or not. But the second phase of the dynamics is equivalent to matching $B_{T_1^{(n)}}^{(n)}$ degree $1$ vertices and then removing these vertices and the edges so formed. In the same way as was discussed in \cite{janson2009percolation}, matching degree $1$ vertices and then removing them and their edges from the graph does not modify the criticality of a Configuration Model graph with finite second moment\footnote{This is so because, for each degree $1$ vertex removed, the largest connected component's size is reduced at most by $1$. This is not true for larger degree vertices.}. We can then establish the subcriticality of the graph after applying $M_1^{(n)}(\cdot)$ just by computing the criticality parameter $\tilde \nu$ of a graph of limiting degree distribution $(\tilde p_k)_{k\geq1}$.

By explicitly computing the criticality parameter we obtain that

\begin{equation}
    \tilde \nu = \frac{\sum_{i\geq2}i (i-1) \tilde p_i}{\sum_{i\geq1}i \tilde p_i} = \frac{1}{\lambda} \sum_{i\geq2} i(i-1) Q^{i-2} p_i = G_D''(Q)/\lambda
\end{equation}

By Theorem 2.3 \cite{janson2009new}, the obtained graph is subcritical when this parameter is strictly less than $1$. The conclusion then follows by Proposition \ref{prop:optimalidadCM}\footnote{At first sight, one could only apply this proposition for distributions with asymptotically exponentially thin tails. But, as shown in Lemma \ref{lemma:colaexp}, this will be true for every degree distribution after applying the map $M_1(\cdot)$ one time.}.

\subsection{Proof of Theorem \ref{thm:furtherapp}}\label{sec:proof22}
    
    The structure of the remaining of the proof the following. We will first give the description of the stochastic process associated to the \emph{second phase} of the dynamics. After this, we use the results in \ref{sec:fluidres} to establish hydrodynamic limits for this process. Finally, we use this limits to prove the statement of the theorem.
    \\\\
    \emph{(i) Stochastic description of the second phase:} this phase of the dynamics consists of sequentially matching the half-edges of blocked vertices and removing the edges so formed. This is done until no more unmatched blocked half-edges remain. In this phase, the states of vertices are not changed, only their degrees; and so, no vertex is added to the independent set.
    
    Initially, we have that graph has $B_{T_1^{(n)}}^{(n)}$ unmatched blocked half-edges and (for $k\geq2$) $\mu^{(n)}_{T_1^{(n)}}(k)$ unexplored vertices of degree $k$. Because of the results of step (ii) of the proof of Theorem \ref{thm:oneapplication} we have that $B_{T_1^{(n)}}^{(n)}/n\xrightarrow{\mathbb{P}}Q^2 \lambda -\sum_{i\geq2}iQ^ip_i$ and (for every $k\geq2$) $\mu_{T_1^{(n)}}^{(n)}(k)/n\xrightarrow{\mathbb{P}} Q^k p_k$.

    Then, at each time $t \geq 0$ the state of the process will be described by the infinite dimensional vector $(U_t^{(n)},B_t^{(n)},\mu_t^{(n)}(0),\mu_t^{(n)}(1),...)$. The process in question will evolve as follows: at each time $t \geq 0$ every unmatched blocked half-edge will have an exponential clock with rate $U_t^{(n)}/B_t^{(n)}$ (and, when $B_t^{(n)}=0$, we define the transition rate as $0$). When one of the clocks of some of this half-edges rings, it is uniformly matched to another free half-edge and the edge formed is removed from the graph. The process will go on until there are no more free blocked half-edges to be paired. Because it is defined by well-behaved transition rates, it is straightforward that the resulting process is Markovian.
    \\\\
\emph{(ii) Hydrodynamic limit of the second phase:} here we establish the convergence of the process associated to the second phase of the dynamics towards the solutions of a set of differential equations. As before, we first find the drifts associated to each one of the coordinates of the state vector:

    \begin{itemize}
        \item With rate $U_t^{(n)}$, the clock of one of the free blocked half-edges rings, at which point it is paired to another free half-edge. So, $U_t^{(n)}$ has a drift given by $\delta(U_t^{(n)}):= - 2 U_t^{(n)}$.
        \item With rate $U_t^{(n)}$, the clock of one of the free blocked half-edges rings, at which point two things can happen: with probability $\frac{B_t^{(n)}}{U_t^{(n)}}$ the half-edge is matched to another free blocked half-edge and therefore $B_t^{(n)}$ is reduced by $2$; or with probability $\frac{U_t^{(n)}-B_t^{(n)}}{U_t^{(n)}}$ is matched to a half-edge belonging to an unexplored vertex and $B_t^{(n)}$ is only reduced by $1$. So, $B_t^{(n)}$ has a drift given by $\delta'(U_t^{(n)},B_t^{(n)}):= -( B_t^{(n)} + U_t^{(n)} )$.
        \item Finally, with rate $U_t^{(n)}$, the clock of one of the free blocked half-edges rings and if matched to an unexplored half-edge, an unexplored vertex is selected according to the size-biased distribution and has one of its half-edges removed. This means that (for each $k\geq0$) with probability $\frac{k \mu_t^{(n)}(k)}{U_t^{(n)}}$ a degree $k$ vertex is selected to be matched and therefore $\mu_t^{(n)}(k)$ is reduced by $1$ and $\mu_t^{(n)}(k-1)$ is increased by $1$. So, $\mu_t^{(n)}(k)$ has a drift given by $\delta_k(\mu_t^{(n)}(k)) := -k \mu_t^{(n)}(k) + (k+1) \mu_t^{(n)}(k+1)$.
    \end{itemize}
    
    Fix some $\delta > 0$, then the sequence of processes $\frac{U_t^{(n)}}{n}$, $\frac{B_t^{(n)}}{n}$ and (for every $k\geq 0$) $\frac{\mu_t^{(n)}(k)}{n}$, are (for $n$ large enough) uniformly bounded by $\lambda + \delta$. They are then uniformly bounded. This, in turn, means that all the drifts associated to the coordinates of  $(U_t^{(n)},B_t^{(n)},\mu_t^{(n)}(0),\mu_t^{(n)}(1),...)$ are uniformly bounded.
    
    By the results shown in the proof of Theorem \ref{thm:oneapplication}, $U_0^{(n)}/n \rightarrow \lambda Q^2$, $B_0^{(n)}/n \rightarrow \lambda Q^2 - \sum_{i\geq2} i Q^i p_i$ and (for every $k\geq0$) $\mu_0^{(n)}(k)/n \rightarrow Q^k p_k \mathbb{1}_{k\geq2}$.
    
    Here we will denote the Dynkin's martingales associated to $U_t^{(n)}$, $B_t^{(n)}$ and (for each $k\geq0$) $\mu_t^{(n)}(k)$ by $M_t^{(n)}$, $M'^{(n)}_t$ and $N_t^{(n)}(k)$, respectively. These are all martingales of locally finite variation. Therefore, their quadratic variation will be given by
    
    \begin{equation}
        [M^{(n)}_t]_t = \sum_{0\leq s \leq t}(\Delta M_s^{(n)})^2 = \sum_{0\leq s \leq t}(\Delta U_s^{(n)})^2 \leq \sum_{s \geq 0}(\Delta U_s^{(n)})^2 \leq 4 B_0^{(n)} n = \mathcal{O}(n)
    \end{equation}
    
    \begin{equation}\label{eq:quadvarB2}
        [M'^{(n)}_t]_t = \sum_{0\leq s \leq t}(\Delta M'^{(n)}_s)^2 = \sum_{0\leq s \leq t}(\Delta B_s^{(n)})^2 \leq \sum_{s \geq 0}(\Delta B_s^{(n)})^2 \leq 2 B_0^{(n)} n = \mathcal{O}(n)
    \end{equation}
    
    \begin{equation}
        [N^{(n)}_t(k)]_t = \sum_{0\leq s \leq t}(\Delta N^{(n)}_s(k))^2 = \sum_{0\leq s \leq t}(\Delta \mu_s^{(n)}(k))^2 \leq \sum_{s \geq 0}(\Delta \mu_s^{(n)}(k))^2 \leq Q^k p_k^{(n)} n = \mathcal{O}(n)
    \end{equation}
    
    Now, the corresponding system of integral equations (as presented in Lemma \ref{lema:limitefluido}) will be given by:

\begin{equation}\label{eq:eqdifU2}
    u_t = \lambda - \int_0^t 2 u_s ds
\end{equation}

\begin{equation}\label{eq:eqdifB2}
    b_t = b_0 - \int_0^t (b_s + u_s) ds
\end{equation}

\begin{equation}\label{eq:eqdifMu2}
    \mu_t(k) = Q^k p_k \mathbb{1}_{\{k\geq2\}} + \int_0^t (k+1) \mu_s(k+1) - k \mu_s(k) ds
\end{equation}

The first two equations can be directly integrated to give

\begin{equation}\label{eq:solU2}
    u_t = \lambda Q^2 e^{-2 t}
\end{equation}

\begin{equation}\label{eq:solB}
    b_t = Q^2 \lambda e^{-2t} - e^{-t} \sum_{i\geq2} i Q^i p_i
\end{equation}

While for equations \eqref{eq:eqdifMu2} with $k\geq1$, they can be seen to have normal modes given by

\begin{equation}
    \eta_k(i) = (-1)^{k-i} \binom{k}{i} \mathbb{1}_{\{i\leq k\}}
\end{equation}

where $i\geq1$ and with associated eigenvalues $\omega_k = -k$.

Furthermore, the uniqueness of these solutions can be proved by decoupling the system, and writing it in the base of the normal modes. This results in a countable number of independent equations with Lipschitz derivatives and the uniqueness then follows by standard ODE theory.

We are thus under the conditions of Lemma \ref{lema:limitefluido} and we will therefore have that, uniformly in compact sets, $U_t^{(n)}/n\xrightarrow{\mathbb{P}}u_t$, $B_t^{(n)}/n\xrightarrow{\mathbb{P}}b_t$ and (for every $k\geq0$) $\mu_t^{(n)}(k)/n\xrightarrow{\mathbb{P}}\mu_t(k)$.

Moreover, if we define the stopping time $T_2^{(n)} := \inf\{t\geq0: B_t^{(n)}=0\}$ and the deterministic time $T_2 := \inf\{t\geq0:b_t=0\}$, it can be proved that the conditions of Lemma \ref{lema:convtiempo} hold. For this, first note that $B_t^{(n)}$ is in fact decreasing (it only transitions to states of lower value) and that for every $i\in\mathbb{N}$ if $B_t^{(n)} = i$ then the transition rate to lower states is given by

\begin{equation*}
    \sum_{j\leq i} Q_{ij} = E_t^{(n)} + B_t^{(n)} = U_t^{(n)}
\end{equation*}

Because for every $t \leq T_2^{(n)}$ we have that $U_t^{(n)} \geq U_0^{(n)} - 2 B_0^{(n)}$, if the initial proportion of blocked vertices if smaller than $1/2$ then $U_t^{(n)}$ will be uniformly lower bounded by a positive number. But because of Lemma \ref{lemma:caemitad} of Appendix \ref{app:urn}, the proportion of blocked vertices will drop w.h.p. (for any initial value) below $1/2$ at a time where $U_t^{(n)}$ is still a positive proportion of $n$. And so, the strict positivity of $U_t^{(n)}/n$ will still be true. 

Furthermore, if we define $\tilde \lambda := \sum_{i\geq2} i Q^i p_i$, $T_2$ can be explicitly found to be given by $\log( Q^2 \lambda / \tilde \lambda) = - \log \tilde Q$. Since (for every $t\leq T_2$)

\begin{equation*}
    \dot b_t = -2 Q^2 \lambda e^{-2t} + e^{-t} \tilde \lambda \leq -2 Q^2 \lambda e^{-T_2} + \tilde \lambda = \tilde \lambda \left[ 1 - \frac{2 \tilde \lambda}{Q^2 \lambda} \right]    
\end{equation*}

where the r.h.s. is strictly negative, then we can apply Lemma \ref{lema:convtiempo} to show that $T_2^{(n)}\xrightarrow{\mathbb{P}}T_2$. And by Corollary \ref{coro:convcoord}, we will have that

\begin{equation*}
    U_{T_2^{(n)}}/n \xrightarrow{\mathbb{P}} u_{T_2}
\end{equation*}
\begin{equation*}
    B_{T_2^{(n)}}/n \xrightarrow{\mathbb{P}} b_{T_2}
\end{equation*}
\begin{equation*}
    \mu_{T_2^{(n)}}(1)/n \xrightarrow{\mathbb{P}} \mu_{T_2}(1)
\end{equation*}
\begin{equation*}
    \mu_{T_2^{(n)}}(2)/n \xrightarrow{\mathbb{P}} \mu_{T_2}(2)
\end{equation*}
\begin{equation*}
    . \ . \ .
\end{equation*}

Finally, because $\sum_{k\geq1} k^2 p_k^{(n)}$ converges towards $\sum_{k\geq1} k^2 p_k$, then these sums are uniformly summable. And because by Lemma \ref{lemma:colaexp}, we have that $\mu_{T_2^{(n)}}^{(n)}(k)/n$ is $\mathcal{O}_\mathbb{P}(e^{-\gamma k})$ (for some $\gamma > 0$), the sums $\sum_{k\geq2}k^2 \mu^{(n)}_{T_2^{(n)}}/n$ will be w.h.p. eventually uniformly summable. This implies that

\begin{equation*}
    \sum_{k\geq1} k^2 p_k^{(n)}\mu^{(n)}_{T_2^{(n)}}/n\xrightarrow{\mathbb{P}} \sum_{k\geq2} k^2 \mu_{T_2}(k)    
\end{equation*}

We then recover the (CA) for the degree distribution of the remaining graph after the second phase. This means that the graph obtained when the second phase is finished can be treated as a Configuration Model with limiting degree distribution (for $k\geq1$) $\tilde p_k = \mu_{T_2}(k)/Z$, where $Z$ is just a normalisation constant.

Finally, we show that the condition of exponentially thin tails in Proposition \ref{prop:optimalidadCM} is not really restrictive, as every initial distribution results in a distribution with this property after one application of the map $M^{(n)}_1(\cdot)$.

\begin{lemma}\label{lemma:colaexp}
    Under the same hypothesis of Theorem \ref{thm:oneapplication} and further assuming that $p_1 > 0$, then $M_1(p_i)(k)$ is $\mathcal{O}(e^{-\gamma k})$ with $\gamma := - log( Q ) > 0$.
\end{lemma}

\begin{proof}
    Recall that $M_1(p_i)(k)$ is given by the $k$-th coordinate of the solution of the system \eqref{eq:eqdifMu2} at time $T_2$. It can then be seen to be smaller or equal to the $k$-th coordinate of the solution of the modified system:
    
    \begin{equation*}
        \begin{cases}
            \tilde \mu_t(i) = Q^i p_i + \int_0^t (i+1) \tilde \mu_s(i+1) - i \tilde \mu_s(i) ds, & \mbox{if } i > k \\
            \tilde \mu_t(i) = Q^i p_i + \int_0^t (i+1) \tilde \mu_s(i+1) ds, & \mbox{if } i = k \\
            \tilde \mu_t(i) = Q^i p_i, & \mbox{if } 1 \leq i < k \\
        \end{cases}  
    \end{equation*}
    
    at time $T_2$.
    
    Furthermore, the $k$-th coordinate of this system can be easily shown to converge monotonically $\tilde \mu_t(k) \nearrow \sum_{i=k}^\infty \tilde \mu_0(i)$ as $t \rightarrow \infty$. Then,
    
    \begin{equation*}
        \mu_{T_2}(k) \leq \tilde \mu_{T_2}(k) \leq \sum_{i=k}^\infty \tilde \mu_0(i) = \sum_{i=k}^\infty \mu_0(i) \leq C \sum_{i=k}^\infty Q^k = \mathcal{O}(Q^k)
    \end{equation*}
    
    where $C>0$ is some constant.
    
\end{proof}

\section{Possible extensions}

In this work we showed that, for a random graph with given degrees, if the degree-greedy algorithm selects only degree 1 or 0 vertices until the remaining graph is subcritical, then the independent set obtained by it is of the same size as a maximum one up to an error term smaller than any positive power of the graph size. We then characterised for which asymptotic degree distributions this happens and gave a way of computing their independence ratio.

It is still an open issue to show if the independent set found is always maximum asymptotically a.s. as in the Erdös-Rényi case; and if not, under which conditions it is.

In Section \ref{sec:exandapp} we explained how, changing higher degree vertices by degree 1 vertices, upper bounds can be obtained for the independence number of general graphs. It would be possible, in principle, to obtain tighter bounds by finding an optimal way of dominating the studied graphs by a graph in which the degree-greedy algorithm is asymptotically optimal.

Furthermore, Lemma \ref{lemma:colaexp} seems to suggest that the pairing of degree 1 vertices quickly generates an exponential tail in the resulting degree distribution. We then conjecture that the condition of finite second moment in Theorems \ref{thm:oneapplication} and \ref{thm:furtherapp} could in fact be avoided, extending the result to heavy-tailed distributions.

Finally, the Glauber dynamics' invariant measure is known (under certain limits) to concentrate around maximum independent sets. Nevertheless, when characterised in this limit, the mixing times are exponential in the graph size. The results of this work might help in showing that the mixing time could be reduced by starting the dynamics from an independent set found by a degree-greedy algorithm.

\bibliographystyle{abbrv}

\bibliography{main}

\appendix

\section{Appendix: Pairing urn model}\label{app:urn}

Suppose we have the following urn problem, which we will refer to as the \emph{Pairing urn model}. Initially we have an urn with $k$ red balls and $n-k$ white balls. At each step of the process, a red ball is removed and after that a second ball is chosen uniformly from the urn and is also removed. The process is continued until there are no more red balls left. What we will prove here is that, if $k(n)= X_0 n + o(n)$ (for some $0<X_0<1$), then the proportion of red balls drops w.h.p. below $1/2$ in a time where there are still a positive proportion of the balls still in the urn. We will denote by $(R_i)_{i\in\mathbb{N}}$ the process that gives at each step $i\geq1$ the number of red balls removed so far. We will also define the stopping time $T : = \inf \{ i \geq0 : (k-R_i) / (n- 2i) < 1/2 \}$ when the proportion of red vertices drops below $1/2$.

\begin{lemma}\label{lemma:caemitad}
    Let $0 < X_0 < 1$ be the asymptotic initial proportion of red balls in a pairing urn process of urn size $n\in\mathbb{N}$, then w.h.p. $T < n/2$.
\end{lemma}
\begin{proof}
We will compare $(R_i)_{i\in\mathbb{N}}$ to a second process $(\tilde R_i(l))_{i\in\mathbb{N}}$, where $l(n)$ is some function of $n$ to be fixed later. At each step $i\geq1$, $\tilde R_i(l)$ will give the total number of red balls removed so far from an urn without replacement with initially $k-l$ red and $n-k$ white balls. We will denote by $X_i$ and $\tilde X_i(l)$ the corresponding proportion of red balls in the urns for both processes.

At each time $j\geq1$ the probability of drawing a red ball for the pairing urn is given by a Bernoulli r.v. of parameter $X_j=(k-R_j)/(n-2j)$ and the corresponding probability for the urn without replacement will also be a Bernoulli r.v. of parameter $\tilde X_j(l)=(k-l-\tilde R_j(l))/(n-l-j)$. We can then couple both selection probabilities by the usual coupling for two Bernoulli variables. Defining $T_l := \inf\{ i\geq1 : X_i \leq \tilde X_i(l) \}$, we will then have that, for every $i \leq T_l$, $R_i - i \geq \tilde R_i(l)$. Now, suppose that $T_l > l$, then at step $l$ we will have that $R_l - l \geq \tilde R_l(l)$ which implies that the proportion of vertices obeys

\begin{equation}
    X_l = \frac{k - R_l}{n-2l} \leq \frac{k - l - \tilde R_l(l)}{n - 2l} = \tilde X_l(l)
\end{equation}

Which contradicts the hypothesis that $T_l > l$. We will then have that $T_l \leq l$ a.s. At each step $i\geq1$, the corresponding value of $\tilde X_i(l)$ will be distributed according to

\begin{equation}
    \mathbb{P}(\tilde R_i(l) = n-l-m) = \frac{\binom{k-i}{(k-i)-m}\binom{n-k}{m-(k-2i)}}{\binom{n-i}{i}}
\end{equation}

By \cite{hoeffding1963probability} we will have that for $i \leq l$, the probability that the proportion of vertices deviating $\delta > 0$ from its mean (which is its initial value $\tilde X_0(l) = (k-l)/(n-l)$) will be upper bounded by

\begin{equation}\label{eq:ldurna}
    \mathbb{P}( \tilde X_i(l) - \tilde X_0(l) > \delta ) \leq \left[ \left( \frac{\tilde X_0(l)}{\tilde X_0(l)+q(n,\delta)} \right)^{\tilde X_0(l)+q(n,\delta)} \left( \frac{1-\tilde X_0(l)}{1- \tilde X_0(l)-q(n,\delta)} \right)^{1-\tilde X_0(l)-q(n,\delta)} \right]^n
\end{equation}

Where $q(n,\delta) := (1 - 2l/n) \delta$. In particular, if asymptotically $2 l(n) < n$, we will then have that

\begin{equation}
    \mathbb{P}\left( \sup_{i\leq l} \tilde X_i - \tilde X_0 \geq \delta \right) \xrightarrow{n\rightarrow\infty} 0 
\end{equation}

And because $T_l \leq l$, fixing $l = 2 k + \delta' - n$ (where $\delta'>0$), gives $\tilde X_0(l) < 1/2$; which implies that the proportion of red balls for the first process will drop in a finite time below $1/2$. Further more, if there exists a $\delta'$ s.t. $2 l < n$, the number of remaining balls in the urn at time $l$ will be a positive proportion of $n$ (i.e., $T < n/2$) as in each step exactly two balls are removed. It is easy to check that there will exist such $\delta'$ whenever $X_0 < 3/4$.

Lets see that when the initial proportion of red vertices is higher than $3/4$ our claim is still true. To prove this, we will work inductively. Call $a_1 := 1/2$ and (for every $i \geq 2$) $a_i := 1 - (2/3)^{i}$. We will suppose that initially $X_0 \in [a_j,a_{j+1})$. The process then arrives at the interval $[a_{j-1},a_j)$ in a finite time and with a positive proportion of $n$ of balls left in the urn. This will follow from the same argument as above by fixing $l(n) = ((X_0 - a_j) /(1- a_j) + \delta'') n$, for some $\delta''> 0$ small enough. This value of $l(n)$ will then give that $\tilde X_0(l) < a_j$ and $l(n) < (1/3+ \delta'') n$, assuring that the same reasoning as before can be used. Because these intervals are a partition of $(1/2,1)$, if the pairing urn starts with any initial proportion of red balls higher than $1/2$, it will eventually (after going through a finite number of intervals) drop below $1/2$ at a time where there are still a positive proportion of balls in the urn. Which is equivalent to say that $T < n/2$. \end{proof}

\end{document}